\documentclass[a4paper]{article}
\usepackage{amsmath,bm}
\usepackage{amssymb}
\usepackage{amsthm}
\usepackage{mathtools}
\usepackage{mathrsfs}
\usepackage[utf8]{inputenc}
\usepackage{cases}
\usepackage{graphicx,color}
\usepackage{float}
\usepackage{kantlipsum}
\usepackage{enumitem}
\usepackage[numbers]{natbib}
\usepackage{tikz}
\usetikzlibrary{intersections,calc,arrows}

\theoremstyle{plain}
\newtheorem{thm}{Theorem}[section]
\newtheorem{lem}[thm]{Lemma}
\newtheorem{cor}[thm]{Corollary}
\newtheorem{prop}[thm]{Proposition}

\theoremstyle{definition}
\newtheorem{df}[thm]{Definition}
\newtheorem{eg}[thm]{Example}

\theoremstyle{remark}
\newtheorem{rem}[thm]{Remark}

\newcommand{\qad}{\phantom{={}}}

\numberwithin{equation}{section}
\allowdisplaybreaks[3]

\newcommand{\E}{\mathbb{E}}

\newcommand{\M}{\mathbb{M}}
\newcommand{\N}{\mathbb{N}}
\renewcommand{\P}{\mathbb{P}}

\newcommand{\R}{\mathbb{R}}

\newcommand{\Z}{\mathbb{Z}}
\newcommand{\1}{\mbox{\rm1}\hspace{-0.25em}\mbox{\rm l}}

\newcommand{\cA}{\mathcal{A}}

\newcommand{\cC}{\mathcal{C}}

\newcommand{\cK}{\mathcal{K}}

\newcommand{\cN}{\mathcal{N}}

\newcommand{\cW}{\mathcal{W}}

\newcommand{\bfk}{\mathbf{k}}

\renewcommand{\a}{\alpha}

\newcommand{\dl}{\delta}

\newcommand{\eps}{\varepsilon}

\newcommand{\lm}{\lambda}
\newcommand{\om}{\omega}
\newcommand{\Om}{\Omega}
\newcommand{\sg}{\sigma}
\newcommand{\Sg}{\Sigma}

\newcommand{\del}{\partial}

\DeclareMathOperator{\cov}{Cov}

\DeclareMathOperator{\Lip}{Lip}

\DeclareMathOperator{\sgn}{sgn}

\DeclareMathOperator{\Supp}{Supp}
\DeclareMathOperator{\Sym}{Sym}
\DeclareMathOperator{\Tr}{Tr}

\DeclareMathOperator{\up}{up}
\DeclareMathOperator{\Var}{Var}

\title{Central limit theorem for linear eigenvalue statistics of the adjacency matrices of random simplicial complexes}
\author{Shu \textsc{Kanazawa}\thanks{Kyoto University Institute for Advanced Study, Kyoto University, Kyoto 606-8501, Japan;
Department of Mathematics, The Ohio State University, Columbus 43210, USA. \texttt{kanazawa.shu.2w@kyoto-u.ac.jp}} and Khanh Duy \textsc{Trinh}\thanks{Global Center for Science and Engineering, Waseda University, Tokyo 169-8555, Japan. \texttt{trinh@aoni.waseda.jp}}}
\date{}

\begin{document}
\maketitle

\begin{abstract}
We study the adjacency matrix of the Linial--Meshulam complex model, which is a higher-dimensional generalization of the Erd\H os--R\'enyi graph model.
Recently, Knowles and Rosenthal proved that the empirical spectral distribution of the adjacency matrix is asymptotically given by Wigner's semicircle law in a diluted regime.
In this article, we prove a central limit theorem for the linear eigenvalue statistics for test functions of polynomial growth that is of class $C^2$ on a closed interval.
The proof is based on higher-dimensional combinatorial enumerations and concentration properties of random symmetric matrices.
Furthermore, when the test function is a polynomial function, we obtain the explicit formula for the variance of the limiting Gaussian distribution.
\end{abstract}

\providecommand{\keywords}[1]
{
  \small	
  \textbf{Keywords } #1
}
\providecommand{\subjclass}[1]
{
  \small	
  \textbf{Mathematics Subject Classification } #1
}




\section{Introduction}

\subsection{Background and prior work}
Random graph theory has been providing a good understanding of large complex systems such as social and biological networks, where each vertex and edge represent an object and a connection between two individual objects, respectively.
The systematic study of random graphs has its origin in the work of Erd\H os and R\'enyi~\cite{ER59,ER60}.
Given $n\in\N$ and $p\in[0,1]$, an Erd\H os--R\'enyi graph $G_{n,p}$ is a random graph on $[n]\coloneqq\{1,2,...,n\}$, where each edge beween two vertices is placed independently of others with probability $p$.
There has been considerable study of the spectrum of the adjacency matrix and the Laplacian matrix of $G_{n,p}$ (see, e.g.,~\cite{BG01,BL10,BLS11,Coja07,DJ10,EKYY12,EKYY13,FO05,Gu21,HKP21,HLY15,Ji12,JL18,KKM06,SB10,SB12,TVW13}).
We review a classical result on the asymptotic distribution of the spectrum of the adjacency matrix of $G_{n,p}$ for a fixed $p\in(0,1)$ and $n$ tending to infinity.
Let $A(G_{n,p})$ denote the adjacency matrix of $G_{n,p}$.
We note that it is more convenient to consider an appropriately centered and scaled adjacency matrix:
\[
H(G_{n,p})\coloneqq\frac1{\sqrt{np(1-p)}}(A(G_{n,p})-\E[A(G_{n,p})]).
\]
Under this normalization, for fixed $p\in (0,1)$, the matrix $H(G_{n,p})$ is a (classical) Wigner matrix and thus Wigner's semicircle law holds, that is, the empirical distribution of the eigenvalues converges weakly to the standard semicircle distribution almost surely.

Let us introduce Wigner's semicircle law in more details. We denote the eigenvalues of $H(G_{n,p})$  by
\[
\lm_1[H(G_{n,p})]\ge\lm_2[H(G_{n,p})]\ge\cdots\ge\lm_n[H(G_{n,p})],
\]
and define their empirical distribution, called the \textit{empirical spectral distribution} of $H(G_{n,p})$, by
\[
L_{H(G_{n,p})}\coloneqq\frac1n\sum_{i=1}^n\dl_{\lm_i[H(G_{n,p})]}.
\]
Here, $\dl_\lm$ indicates the Dirac measure at $\lm$, i.e., for any Borel set $B\subset\R$,
\[
\dl_\lm(B)\coloneqq\begin{cases}
1	&\text{if $\lm\in B$,}\\
0	&\text{if $\lm\notin B$.}
\end{cases}
\]
Wigner's semicircle law states that with probability one, the empirical spectral distribution $L_{H(G_{n,p})}$ converges weakly to the standard semicircle distribution 
$
\nu(dx)\coloneqq (2\pi)^{-1}{\sqrt{4-x^2}}\1_{\{|x|\le2\}}dx.
$
Since the semicircle distribution $\nu$ has compact support, the law is a consequence of the following almost sure convergence for moments: for any $k = 0,1,\dots,$
\[
	\langle L_{H(G_{n,p})}, x^k\rangle
    \coloneqq\int_\R x^k\,dL_{H(G_{n,p})}(x)
    = \frac{1}{n} \sum_{i=1}^n (\lm_i[H(G_{n,p})])^k \xrightarrow[n\to\infty]{} \int_{-2}^2 x^k\,d\nu(x) \quad \text{almost surely.}
\]
Here and in what follows, we use the following abbreviated notation for integrals:
\[
\langle\mu,f\rangle\coloneqq\int_\R f(x)\,d\mu(x)
\]
for any probability measure $\mu$ on $\R$ and integrable function $f\colon\R\to\R$ with respect to $\mu$. We remark that the above convergence still holds when $p$ varies as a function of $n$ with $n p \to \infty$. In the so-called sparse regime, that is, when $np$ is of order one,
the matrix $H(G_{n, p})$ is an example of Wigner matrix with exploding moments and their empirical spectral distribution is known to converge to a different limiting distribution~\cite{BG01,BGM14,BL10,BLS11,JL18,KKM06,KSV04, Ma17,Za06}.
The properties of the limiting distribution is extensively studied in~\cite{AB21,BSV17,CS21,Sa15}.

Gaussian fluctuations around the limit for Wigner matrices have been well-studied. For fixed $p \in (0,1)$, central limit theorems (CLTs) for moments of the empirical spectral distribution have the following form: 
\[
	n(\langle L_{H(G_{n,p})}, x^k\rangle - \E\langle L_{H(G_{n,p})}, x^k\rangle) \xrightarrow[n\to\infty]{d} \cN(0, \sigma_k^2)
\]
for some $\sg_k^2\ge0$.
Here, $\xrightarrow[n\to\infty]{d}$ denotes the convergence in distribution as $n\to\infty$, and $\cN(\mu, \sigma^2)$ denotes the Gaussian distribution with mean $\mu$ and variance $\sigma^2$. The joint convergence also holds.
Consequently, a CLT holds for any polynomial test function $f$, that is, 
\[
n(\langle L_{H(G_{n,p})}, f\rangle - \E [\langle L_{H(G_{n,p})}, f\rangle]) \xrightarrow[n\to\infty]{d} \cN(0, \sigma_f^2)
\]
for some $\sigma_f^2 \ge 0$. Extending polynomial type CLTs to a larger class of test functions and explicit formula for the limiting variance have also been known (see~\cite[Section~2.7]{AGZ11} and reference therein; see also~\cite[Theorem~1]{Sh11}).
In different regimes, we refer the reader to \cite{SB12} in the case that $n p \to \infty$ with $p \to 0$; and to \cite{BGM14, SB10} for the sparse regime.

In this article, we consider random simplicial complexes, which can be regarded as a higher-dimensional generalization of random graphs (see Subsection~\ref{ssec:high_dim_tree} for the precise definition of simplicial complexes).
The study of random simplicial complexes has its origin in the work of Linial and Meshulam~\cite{LM06}.
Given $n\in\N$ and $p\in[0,1]$, they introduced a random $2$-dimensional simplicial complex on $[n]$, later known as the \textit{$2$-Linial--Meshulam complex}, where all the edges between two vertices are placed in advance (the underlying graph is the complete graph $K_n$) and each triangle formed by three vertices is filled by a face independently of others with probability $p$.
More generally, the \textit{$d$-Linial--Meshulam complex} was introduced by Meshulam and Wallach~\cite{MW09} for each $d\in\N$.
In the following, let $\cK_n\coloneqq2^{[n]}$ denote the complete complex on $[n]$, which has all the possible simplices in $[n]$.
Furthermore, for $0\le i<n$, let $\cK_n^{(i)}$ denote the $i$-skeleton of $\cK_n$, that is, the subcomplex of $\cK_n$ consisting of all the simplices in $\cK_n$ of dimension at most $i$.
\begin{df}[$d$-Linial--Meshulam complex]
Let $d,n\in\N~(n>d)$ and $p\in[0,1]$.
The \textit{$d$-Linial--Meshulam complex $Y^d_{n,p}$} is a random subcomplex of $\cK_n$ given by
\[
\P(Y^d_{n,p}=Y)=\begin{cases}
p^{f_d(Y)}(1-p)^{\binom n{d+1}-f_d(Y)}	&\text{if $\cK_n^{(d-1)}\subset Y\subset\cK_n^{(d)}$,}\\
0							&\text{otherwise}
\end{cases}
\]
for any subcomplex $Y\subset\cK_n$.
Here, $f_d(Y)$ indicates the number of $d$-simplices in $Y$.
Note that the $1$-Linial--Meshulam complex $Y^1_{n,p}$ can be regarded as the Erd\H os--R\'enyi graph $G_{n,p}$.
\end{df}

It is a natural question to ask whether analogous results to Wigner's semicircle law and Gaussian fluctuations hold for $d$-Linial--Meshulam complexes.
In order to answer this question, we briefly explain the notion of the higher-dimensional adjacency matrices for simplicial complexes, which was introduced in~\cite{GW12} (see Subsection~\ref{ssec:adjacency_matrix} for the precise definition).
Given a simplicial complex $X$ of dimension at least $k\ge0$, the $k$th adjacency matrix, denoted by $A_k(X)$, is a matrix whose rows and columns are indexed by the $k$-simplices in $X$.
The matrix $A_k(X)$ has entries in $\{0,\pm1\}$, where each entry is nonzero if and only if two distinct $k$-simplices are contained in a common $(k+1)$-dimensional simplex, and in such cases, the sign of the entry is determined according to the orientation of the two $k$-simplices.
The adjacency matrix $A_k(X)$ has a simple connection to the combinatorial up Laplacian $L_k^{\text{up}}(X)$ for simplicial complexes, introduced in~\cite{Ec45}: $A_k(X)=D_k(X)-L_k^{\text{up}}(X)$, where $D_k(X)$ is the degree matrix of $k$-simplices (the degree of a $k$-simplex is the number of its $(k+1)$-dimensional cofaces).
Note that the zeroth adjacency matrix coincides with the usual adjacency matrix of the underlying graph.

We are interested in the $(d-1)$st adjacency matrices $A_{d-1}(Y^d_{n,p})$ of $d$-Linial--Meshulam complexes $Y^d_{n,p}$.
Although the appearance of each $d$-simplex is independent of others, the above-diagonal entries of $A_{d-1}(Y^d_{n,p})$ are no longer independent whenever $d\ge2$.
Indeed, since each $d$-simplex has $d+1$ number of $(d-1)$-dimensional faces, the appearance of the $d$-simplex affects $\binom{d+1}2$ entries above the diagonal.
In this sense, the adjacency matrix of random simplicial complexes is an important model by its own from the point of view of the random matrix theory.
We again consider an appropriately centered and scaled adjacency matrix:
\[
H_{d-1}(Y^d_{n,p})
\coloneqq\frac1{\sqrt{np(1-p)}}(A_{d-1}(Y^d_{n,p})-\E[A_{d-1}(Y^d_{n,p})]).
\]
In the following, we regard $p=p(n)$ as a function of $n$, and write $H_n
=H_{d-1}(Y^d_{n,p})$ for simplicity.
Recently, Knowles and Rosenthal~\cite{KR17} proved the following theorem, which is a higher-dimensional generalization of Wigner's semicircle law for the centered adjacency matrices of the random graph $G_{n, p}$.
\begin{thm}[{\cite[Theorem~3.1]{KR17}}]\label{thm:KR17_thm3.1}
If $\lim_{n\to\infty}np(1-p)=\infty$, then the empirical spectral distribution $L_{H_n}$ converges weakly to $\nu_d$ almost surely.
Here,
\[
\nu_d(dx)\coloneqq\frac{\sqrt{4d-x^2}}{2\pi d}\1_{\{|x|\le2\sqrt d\}}dx
\]
is a semicircle distribution.
\end{thm}
By an additional argument involving Weyl's interlacing inequalities, the above theorem also holds for $A_{d-1}(Y^d_{n,p})$ itself instead of the centered adjacency operator~\cite[Theorem~2.5]{KR17}.
As pointed out in~\cite[Section~3]{KR17}, Theorem~\ref{thm:KR17_thm3.1} follows from the next two lemmas via the Weierstrass approximation theorem, a standard truncation argument, and the Borel--Cantelli lemma.
Such approach is known as the moment method.
\begin{lem}[{\cite[Lemma~3.2]{KR17}}]\label{lem:KR17_Lem3.2}
If $\lim_{n\to\infty}np(1-p)=\infty$, then for any $k\in\Z_{\ge0}$,
\[
\lim_{n\to\infty}\E\langle L_{H_n},x^k\rangle=\langle\nu_d,x^k\rangle=\begin{cases}
0			&\text{if $k$ is odd,}\\
d^{k/2}\cC_{k/2}	&\text{if $k$ is even.}
\end{cases}
\]
Here, $\cC_k\coloneqq\frac1{k+1}\binom{2k}k$ is the $k$th Catalan number for $k\ge1$, and $\cC_0\coloneqq1$.
\end{lem}
\begin{lem}[{\cite[Lemma~3.3]{KR17}}]\label{lem:KR17_Lem3.3}
If $\lim_{n\to\infty}np(1-p)=\infty$, then for any $k\in\Z_{\ge0}$,
\[
\limsup_{n\to\infty}n^d\{np(1-p)\}\cdot\Var(\langle L_{H_n},x^k\rangle)<\infty.
\]
\end{lem}
For the proof of Lemma~\ref{lem:KR17_Lem3.2}, note first that $\binom nd\langle L_{H_n},x^k\rangle=\Tr(H_n^k)$.
One can expand the trace in terms of the entries of $H_n$ and examine the terms that contribute to the first order.
Knowles and Rosenthal~\cite{KR17} showed that the main contributions to the first order are associated with rooted trees whose edges are labeled by $\{1,2,\ldots,d\}$, explaining why Theorem~\ref{lem:KR17_Lem3.2} involves the Catalan number (see Section~\ref{sec:E_moments} for a more detailed explanation).
As we will see later, Lemma~\ref{lem:KR17_Lem3.3} is a special case of our first main result~(Theorem~\ref{thm:multi_CLT}(1)).

We also remark on several studies on the spectrum of the adjacency matrix and Laplacian of the Linial--Meshulam complex model.
In~\cite{GW16}, Gundert and Wagner proved that, under the assumption $\lim_{n\to\infty}np(1-p)/\log n=\infty$, the eigenvalues of the adjacency operator $A_{d-1}(Y^d_{n,p})$ are asymptotically almost surely confined to two separate intervals $\sqrt{np(1-p)}[-C,C]$ and $np+\sqrt{np(1-p)}[-C,C]$ for some constant $C>0$.
Knowles and Rosenthal~\cite{KR17} improved the size of the intervals to $\sqrt{np(1-p)}[-2\sqrt d-\eps,2\sqrt d+\eps]$ (for any $\eps>0$) and $np+[-7d,7d]$ at the cost of the stronger assumption $\lim_{n\to\infty}np(1-p)/(\log n)^4=\infty$~(see also~\cite[Theorem~2.4]{LR22} for a stronger result).
Furthermore, Leibzirer and Rosenthal~\cite{LR22} studied sparse random matrices as generalizations of the centered and scaled adjacency matrix $H_{d-1}(Y^d_{n,p})$, which is obtained by replacing the Bernoulli distribution with parameter $p$ used to construct the adjacency matrix $A_{d-1}(Y^d_{n,p})$ with any bounded distribution.
Fountoulakis and Przykucki~\cite{FP22} proved a concentration result on the spectral gap of the combinatorial up Laplacian $L_k^{\up}(Y^d_{n,p})$ when $np=(1+\eps)d\log n$ for any $\eps>0$.
When $np$ is of order one, 
Adhikari, Kumar, and Saha~\cite{AKS23} showed the existence of the limiting nonrandom measure of the empirical spectral distribution of $A_{d-1}(Y^d_{n,p})$ (they also considered the unsigned adjacency matrices).
The existence of the limiting nonrandom measure of the empirical spectral distribution of the combinatorial up Laplacian for much larger class of random simplicial complexes was proved in~\cite[Theorem~24]{Ka22}.

\subsection{Main results}
The objective in this article is to investigate Gaussian fluctuations around the limit.
For the main results, recalling that $p=p(n)$ is a function of $n$, we assume that the limit $p_\infty\coloneqq\lim_{n\to\infty}p$ exists in $[0,1]$.
Furthermore, for any $k,l\in\Z_{\ge0}$, we define $\sg(k,l)\in[0,\infty)$ as follows: if $k+l$ is even, then
\begin{align*}
\sg(k,l)
&\coloneqq\1_{\{k,l\colon\text{even}\}}(d+1)!\{(k/2)d^{k/2}\cC_{k/2}\}\{(l/2)d^{l/2}\cC_{l/2}\}(2p_\infty-1)^2\\
&\qad+\sum_{r=3}^{k\wedge l}\frac2r\Biggl(d!\,k\!\sum_{\substack{k_1,k_2,\ldots,k_r\in2\Z_{\ge0}\\k_1+\cdots+k_r=k-r}}\prod_{q=1}^rd^{k_q/2}\cC_{k_q/2}\Biggr)\Biggl(d!\,l\!\sum_{\substack{l_1,l_2,\ldots,l_r\in2\Z_{\ge0}\\l_1+\cdots+l_r=l-r}}\prod_{q=1}^rd^{l_q/2}\cC_{l_q/2}\Biggr)p_\infty(1-p_\infty),
\end{align*}
and $\sg(k,l)\coloneqq0$ otherwise.
\begin{rem}
One can easily verify that $\sg(k,l)>0$ if and only if one of the following holds:
\begin{itemize}
\item $k\wedge l=2$, both $k$ and $l$ are even, and $p_\infty\neq1/2$;
\item $k\wedge l\ge3$ and both $k$ and $l$ are even;
\item $k\wedge l\ge3$, both $k$ and $l$ are odd, and $p_\infty\notin\{0,1\}$.
\end{itemize}
\end{rem}
The following is our first main result.
Recall that $\xrightarrow[n\to\infty]{d}$ denotes the convergence in distribution as $n\to\infty$.
\begin{thm}[Multivariate CLT]\label{thm:multi_CLT}
Suppose that $\lim_{n\to\infty}np(1-p)=\infty$, then the following statements hold.
\begin{enumerate}
\item For any $k,l\in\Z_{\ge0}$,
\[
\lim_{n\to\infty}n^d\{np(1-p)\}\cdot\cov(\langle L_{H_n},x^k\rangle,\langle L_{H_n},x^l\rangle)
=\sg(k,l).
\]
\item Let $K\in\Z_{\ge0}$.
Then, the matrix $\Sg_K\coloneqq\{\sg(k,l)\}_{0\le k,l\le K}$ is symmetric and positive semidefinite, and it holds that
\[
\Bigl\{\sqrt{n^d\{np(1-p)\}}\cdot(\langle L_{H_n},x^k\rangle-\E\langle L_{H_n},x^k\rangle)\Bigr\}_{k=0}^K
\xrightarrow[n\to\infty]{d}\cN(0,\Sg_K).
\]
Here, $\cN(0,\Sg_K)$ denotes the multivariate Gaussian distribution with zero mean vector and covariance matrix $\Sg_K$.
\end{enumerate}
\end{thm}
%
%
The following is an immediate corollary of Theorem~\ref{thm:multi_CLT}.
\begin{cor}[Polynomial-type CLT]\label{cor:poly-type_CLT}
If $\lim_{n\to\infty}np(1-p)=\infty$, then for any real-valued polynomial function $f(x)=\sum_{k=0}^Ka_kx^k$,
\[
\lim_{n\to\infty}n^d\{np(1-p)\}\cdot\Var(\langle L_{H_n},f\rangle)
=\sum_{k,l=0}^Ka_k\sg(k,l)a_l
\]
and it holds that
\[
\sqrt{n^d\{np(1-p)\}}\cdot(\langle L_{H_n},f\rangle-\E\langle L_{H_n},f\rangle)
\xrightarrow[n\to\infty]{d}\cN\biggl(0,\sum_{k,l=0}^Ka_k\sg(k,l)a_l\biggr).
\]
\end{cor}

Next, we lift a polynomial-type CLT to a CLT for a larger class of test functions.
A function $g\colon\R\to\R$ is said to be \textit{of polynomial growth} if there exists a constant $M\ge0$ such that $\lim_{x\to\pm\infty}|g(x)|/|x|^M=0$.
Furthermore, we say that a function $f\colon\R\to\R$ is \textit{of class $C^m$ on a closed interval $[a,b]$} if there exists $\eps>0$ such that $f$ is of class $C^m$ on $(a-\eps,b+\eps)$.
The following is our main theorem, which is a CLT for a test functions $f\colon\R\to\R$ of polynomial growth that is of class $C^2$ on the support of the semicircle distribution $\nu_d$.
\begin{thm}\label{thm:CLT_diff}
If $\lim_{n\to\infty}np(1-p)/(\log n)^4=\infty$, then for any function $f\colon\R\to\R$ of polynomial growth that is of class $C^2$ on $[-2\sqrt d,2\sqrt d]$, there exists a constant $\sg^2\ge0$ such that
\[
\lim_{n\to\infty}n^d\{np(1-p)\}\cdot\Var(\langle L_{H_n},f\rangle)
=\sg^2
\]
and it holds that
\[
\sqrt{n^d\{np(1-p)\}}\cdot(\langle L_{H_n},f\rangle-\E\langle L_{H_n},f\rangle)
\xrightarrow[n\to\infty]{d}\cN(0,\sg^2).
\]
\end{thm}
It is unknown whether we can relax the smoothness assumption on the test function $f$ in the above theorem.
Furthermore, the explicit formula for the variance $\sg^2$ is also an open problem.

The remainder of this article is organized as follows.
In Section~\ref{sec:preliminaries}, we define some important notions, namely pure dimensionality, strong connectedness, and $d$-tree, which repeatedly appear in this article.
We also review the definition of the adjacency matrix for simplicial complexes, introduced in~\cite{GW12}.
Section~\ref{sec:E_moments} presents a brief summary of the proof of Lemma~\ref{lem:KR17_Lem3.2} since the proof includes many important concepts and calculations that are also crucial in the subsequent sections.
The starting point of the proof of Lemma~\ref{lem:KR17_Lem3.2} is to expand $\E\langle L_{H_n},x^k\rangle$ by closed $(n,d)$-words, appropriate walks on $(d-1)$-simplices in $\cK_n$.
In this article, we however introduce a different definition of $(n,d)$-words (Definition~\ref{df:nd_word}), which turns out to be critical for counting specific types of $(n,d)$-sentences, tuples of $(n,d)$-words, in Section~\ref{sec:cov_moments}.
See also Remark~\ref{rem:nd_word} for other definitions of $(n,d)$-words given in~\cite{AKS23,KR17,LR22}.
In Section~\ref{sec:cov_moments}, we prove Theorem~\ref{thm:multi_CLT}(1).
We consider $(n,d)$-sentences to expand $\cov(\langle L_{H_n},x^k\rangle,\langle L_{H_n},x^l\rangle)$.
In order to determine the coefficients of the first order of the expansion, we introduce the notions of bracelets and bracelets with pendant $d$-trees for simplicial complexes (see Definition~\ref{df:bracelet}).
In Section~\ref{sec:multivariate_CLT}, we prove Theorem~\ref{thm:multi_CLT}(2).
By Carleman's theorem, it is important to examine the first order of the expansion of the form $\E\bigl[\prod_{j=1}^h(\langle L_{H_n},x^{k_j}\rangle-\E\langle L_{H_n},x^{k_j}\rangle)\bigr]$ in terms of $(n,d)$-sentences.
In order to do that, we need to reveal the relationship between the number of vertices in the support complex of a $(n,d)$-sentence and the number of its strongly connected components (see Lemma~\ref{lem:cW_ksh}).
To this end, we introduce the notion of $d$-simplex-bounding tables, which is a higher-dimensional analog of the edge-bounding table.
In Section~\ref{sec:CLT_diff_test_funct}, we state a general theorem that lifts a polynomial-type CLT to a CLT for a class of test functions of polynomial growth that is of class $C^2$ on a closed interval~(Theorem~\ref{thm:CLT_poly-diff}).
The notions of the convex concentration property for a sequence of random symmetric matrices is crucial, which is proved by Talagrand's concentration inequality.

\section{Preliminaries}\label{sec:preliminaries}

\subsection{Higher-dimensional tree}\label{ssec:high_dim_tree}
A nonempty finite collection $X$ of finite sets is called a \textit{finite abstract simplicial complex} if $\sg\in X$ and $\tau\subset\sg$ together implies $\tau\in X$.
Throughout this article, we omit the word ``finite'' and ``abstract'', and simply call $X$ a simplicial complex.
Note that all simplicial complexes include the empty set.
Every element $\sg\in X$ is called a \textit{simplex} in $X$, and a (strict) subset $\rho$ of $\sg$ is called a (strict) \textit{face} of $\sg$.
The \textit{vertex set} of $X$, denoted by $V(X)$, is defined as the union of all simplices: $V(X)\coloneqq\bigcup X$.
We often call $X$ a simplicial complex on $V(X)$.
The \textit{dimension} of $\sg\in X$ is defined by $\dim\sg\coloneqq|\sg|-1$, where $|\sg|$ indicates the cardinality of $\sg$.
The \textit{dimension} of $X$, denoted by $\dim X$, is defined as the maximum of the dimensions of the simplices in $X$.
We call $\sg\in X$ with $\dim\sg=k$ a \textit{$k$-simplex} in $X$.
For $k\ge-1$, let $F_k(X)$ be the set of all $k$-simplices in $X$, and $f_k(X)\coloneqq|F_k(X)|$.
A simplicial complex contained in $X$ is called a \textit{subcomplex} of $X$.
Given a simplex $\sg$ in $X$, let $K(\sg)$ be the subcomplex of $X$ consisting of all faces of $\sg$.
\begin{df}\label{df:pure}
Let $X$ be a simplicial complex.
$X$ is called a \textit{pure $d$-dimensional simplicial complex} if for any $\sg\in X$, there exists a $d$-simplex $\tau$ in $X$ containing $\sg$.
A pure $d$-dimensional simplicial complex $X$ is said to be \textit{strongly connected} if for every pair of $(d-1)$-simplices $\sg$ and $\sg'$ in $X$, there exists a sequence $(\sg_0,\sg_1,\ldots,\sg_r)$ of $(d-1)$-simplices in $X$ such that $\sg_0=\sg$, $\sg_r=\sg'$, and $\sg_{i-1}\cup\sg_i\in F_d(X)$ for all $i=1,2,\ldots,r$.
A \textit{strongly connected component} of a pure $d$-dimensional simplicial complex $X$ is an inclusionwise maximal strongly connected subcomplex of $X$.
We denote by $\bar c(X)$ the number of strongly connected components of $X$.
Note that strongly connected components of $X$ cover $X$ but is not necessarily vertex-disjoint, unlike the connected components.
\end{df}

A strongly connected pure $d$-dimensional simplicial complex $X$ can be constructed by a process
\begin{equation}\label{eq:process}
X(0)\subsetneq X(1)\subsetneq\cdots\subsetneq X(f_d(X))=X,
\end{equation}
where we start with the simplicial complex $X(0)$ generated by a $(d-1)$-simplex $\sg_0$ in $X$, i.e., $X(0)=K(\sg_0)$; at each step $t=1,2,\ldots,f_d(X)$, we pick a $(d-1)$-simplex $\sg_{t-1}$ in $X(t-1)$ and a vertex $v_t\in V(X)$, and define $X(t)\coloneqq X(t-1)\cup K(\sg_{t-1}\cup\{v_t\})$.
\begin{df}[$d$-tree]
A pure $d$-dimensional simplicial complex $X$ is called a \textit{$d$-tree} if, in the generating process~\eqref{eq:process}, the vertices $v_1,v_2,\ldots,v_{f_d(X)}$ can be taken as a new vertices at each step, i.e., $v_t\in V(X)\setminus V(X(t-1))$ for all $t=1,2,\ldots,f_d(X)$.
We also regard the $(d-1)$-dimensional simplicial complex generated by a $(d-1)$-simplex as a $d$-tree for convenience.
We call it a \textit{trivial $d$-tree}.
\end{df}
The following lemma provides a characterization of $d$-trees, which is repeatedly used in this article.
\begin{lem}\label{lem:f0-fd}
Let $X$ be a pure $d$-dimensional simplicial complex.
Then,
\[
f_0(X)\le f_d(X)+d\cdot \bar c(X).
\]
The equality holds if and only if every strongly connected component of $X$ is a connected component that is a $d$-tree.
In particular, if $X$ is strongly connected, then $f_0(X)\le f_d(X)+d$, and the equality holds if and if $X$ is a $d$-tree.
\end{lem}
\begin{proof}
We first suppose that $X$ is strongly connected.
At each step in the generating process~\eqref{eq:process}, the number of vertices and $d$-simplices increases by at most one and by exactly one, respectively.
Noting that $f_0(X(0))=f_d(X(0))+d$, we obtain $f_0(X)\le f_d(X)+d$.
Furthermore, the the number of vertices and $d$-simplices both increases by exactly one at all steps if and only if $X$ is a $d$-tree.

Next, we consider the general case.
Let $X_1,X_2,\ldots,X_{\bar c(X)}$ be the strongly connected components of $X$.
We define a new pure $d$-dimensional simplicial complex $X'$ as the disjoint union of $X_1,X_2,\ldots,X_{\bar c(X)}$.
Then, we obviously have $f_0(X)\le f_0(X')$, $f_d(X)=f_d(X')$, and $\bar c(X)=\bar c(X')$.
Since $f_0(X_i)\le f_d(X_i)+d$ for all $i=1,2,\ldots,\bar c(X)$ from the first part, we obtain
\[
f_0(X)\le f_0(X')\le f_d(X')+d\cdot\bar c(X')=f_d(X)+d\cdot\bar c(X).
\]
Both the above equalities hold if and only if every strongly connected component of $X$ is a connected component that is a $d$-tree.
\end{proof}

\subsection{Adjacency operator for simplicial complexes}\label{ssec:adjacency_matrix}
Let $X$ be a simplicial complex.
We fix a total order in $V(X)$.
Given $\tau\in F_{k+1}(X)$ and $k$-dimensional face $\sg$ of $\tau$, we define the \textit{sign between $\tau$ and $\sg$} as follows: if $\tau=\{v_0,v_1,\ldots,v_{k+1}\}$ with $v_0<v_1<\cdots<v_{k+1}$ and $\sg=\{v_0,v_1,\ldots,\hat{v_i},\ldots,v_{k+1}\}$ for some $i=0,1,\ldots,k+1$, then $\sgn(\tau,\sg)\coloneqq(-1)^i$.
Here, the hat symbol over $v_i$ indicates that the vertex is deleted from $\{v_0,v_1,\ldots,v_{k+1}\}$.
Given $\sg,\sg'\in F_k(X)$ such that $\sg\cup\sg'\in F_{k+1}(X)$, we also define the \textit{sign between $\sg$ and $\sg'$} by $\sgn(\sg,\sg')\coloneqq-\sgn(\sg\cup\sg',\sg)\sgn(\sg\cup\sg',\sg')$.
When $k=0$, the sign $\sgn(\sg,\sg')$ is always equal to $1$; however, when $k\ge1$, the sign $\sgn(\sg,\sg')$ depends on the fixed total order in $V(X)$.
\begin{rem}
Let $\sg,\sg'\in F_k(X)$ with $\sg\cup\sg'\in F_{k+1}(X)$.
Then, it holds that
\begin{equation}\label{eq:sign_simplices}
-\sgn(\sg\cup\sg',\sg)\sgn(\sg\cup\sg',\sg')=\sgn(\sg,\sg\cap\sg')\sgn(\sg',\sg\cap\sg').
\end{equation}
Indeed, letting $\sg\cup\sg'=\{v_0,v_1,\ldots,v_{k+1}\}$, we may write
\[
\sg=\{v_0,v_1,\ldots,\hat v_i,\ldots,v_{k+1}\}\text{ and }\sg'=\{v_0,v_1,\ldots,\hat v_j,\ldots,v_{k+1}\}
\]
for some $0\le i<j\le k+1$ without loss of generality.
Then, one can easily verify that both sides of~\eqref{eq:sign_simplices} are equal to $(-1)^{i+j+1}$.
Hence, we can define the sign $\sgn(\sg,\sg')$ by $\sgn(\sg,\sg\cap\sg')\sgn(\sg',\sg\cap\sg')$ instead of $-\sgn(\sg\cup\sg',\sg)\sgn(\sg\cup\sg',\sg')$.
\end{rem}

For $k\ge-1$, let $\Om^k(X)$ be the $\R$-vector space of all \textit{$k$-forms} of $X$:
\[
\Om^k(X)\coloneqq\begin{cases}
\{\varphi\colon F_k(X)\to\R\}   &\text{if $F_k(X)\neq\emptyset$,}\\
0                               &\text{otherwise.}
\end{cases}
\]
The \textit{$k$th adjacency operator} $\cA_k(X)$ of $X$ is a linear operator on $\Om^k(X)$ defined by
\[
(\cA_k(X)\varphi)(\sg)
\coloneqq\sum_{\substack{\sg'\in F_k(X)\\\sg\cup\sg'\in F_{k+1}(X)}}\sgn(\sg,\sg')\varphi(\sg')
\]
for any $\varphi\in\Om^k(X)$ and $\sg\in F_k(X)$.
The \textit{$k$th adjacency matrix} $A_k(X)$ of $X$ is the $F_k(X)\times F_k(X)$ matrix representation of $\cA_k(X)$ with respect to the canonical basis $\{\1_\sg\}_{\sg\in F_k(X)}$ of $\Om^k(X)$: the $(\sg,\sg')$-entry is given by
\[
A_k(X)_{\sg,\sg'}=\begin{cases}
\sgn(\sg,\sg')	&\text{if $\sg\cup\sg'\in F_{k+1}(X)$,}\\
0			&\text{otherwise}
\end{cases}
\]
for any $\sg,\sg'\in F_k(X)$.
A one-dimensional simplicial complex can be naturally regarded as a finite undirected graph with no multiple edges and no self-loops.
In this case, the zeroth adjacency matrix is nothing but the usual adjacency matrix of the graph.

For an $N\times N$-symmetric matrix $A$, we denote its real eigenvalues by
\[
\lm_1[A]\ge\lm_2[A]\ge\cdots\ge\lm_N[A],
\]
and define the empirical spectral distribution $L_A$ of $A$ by
\[
L_A\coloneqq\frac1N\sum_{i=1}^N\dl_{\lm_i[A]}.
\]
Here, recall that $\dl_\lm$ indicates the Dirac measure at $\lm$.
Note that the empirical spectral distribution of $A_k(X)$ does not depend on the choice of the total order in $V(X)$.
Indeed, a change of the total order in $V(X)$ just flips the signs of rows and columns.
\begin{eg}
Let $n>d\ge1$ be fixed, and let $\cK_n$ denote the complete complex on $[n]$.
As shown in~\cite[Lemma~8]{GW16}, $A_{d-1}(\cK_n)$ has eigenvalues $n-d$ and $-d$ with multiplicities $\binom{n-1}{d-1}$ and $\binom{n-1}d$, respectively:
\begin{align*}
\lm_1[A_{d-1}(\cK_n)]=\lm_2[A_{d-1}(\cK_n)]=\cdots=\lm_{\binom{n-1}{d-1}}[A_{d-1}(\cK_n)]&=n-d
\shortintertext{and}
\lm_{\binom{n-1}{d-1}+1}[A_{d-1}(\cK_n)]=\lm_{\binom{n-1}{d-1}+2}[A_{d-1}(\cK_n)]=\cdots=\lm_{\binom nd}[A_{d-1}(\cK_n)]&=-d.
\end{align*}
In particular, the empirical spectral distribution
\[
L_{A_{d-1}(\cK_n)}
=\frac{\binom{n-1}{d-1}}{\binom nd}\dl_{n-d}+\frac{\binom{n-1}d}{\binom nd}\dl_{-d}
\]
converges weakly to $\dl_{-d}$ as $n\to\infty$.
In the proof of~\cite[Lemma~8]{GW16}, the $(d-1)$st combinatorical Laplacian and the Hodge decomposition of $\Om^{d-1}(\cK_n)$ are exploited.
We here provide a straightforward proof for completeness.
For any $1\notin\rho\in F_{d-2}(\cK_n)$, define
\[
\dl^\rho
\coloneqq\sum_{\substack{\sg\in F_{d-1}(\cK_n)\\\sg\supset\rho}}\sgn(\sg,\rho)\1_\sg\in\Om^{d-1}(\cK_n).
\]
Note that $\{\dl^\rho\}_{1\notin\rho\in F_{d-2}(\cK_n)}$ are linearly independent since given $1\notin\rho_0\in F_{d-2}(\cK_n)$, $\dl^{\rho_0}$ is the only $(d-1)$-form taking nonzero at $\rho_0\cup\{1\}$ among $\{\dl^\rho\}_{1\notin\rho\in F_{d-2}(\cK_n)}$.
Furthermore, a straightforward calculation using~\eqref{eq:sign_simplices} yields $\cA_{d-1}(\cK_n)\dl^\rho=(n-d)\dl^\rho$ for any $1\notin\rho\in F_{d-2}(\cK_n)$.
Therefore, $n-d$ is the eigenvalue of $\cA_{d-1}(\cK_n)$ with multiplicity at least $\binom{n-1}{d-1}$.
Next, for any $1\in\tau\in F_d(\cK_n)$, define
\[
\del^\tau
\coloneqq\sum_{\substack{\sg\in F_{d-1}(\cK_n)\\\sg\subset\tau}}\sgn(\tau,\sg)\1_\sg\in\Om^{d-1}(\cK_n).
\]
Note that $\{\del^\tau\}_{1\in\tau\in F_d(\cK_n)}$ are linearly independent since given $1\in\tau_0\in F_d(\cK_n)$, $\del^{\tau_0}$ is the only $(d-1)$-form taking nonzero at $\tau_0\setminus\{1\}$ among $\{\del^\tau\}_{1\in\tau\in F_d(\cK_n)}$.
Furthermore, one can show that $\cA_{d-1}(\cK_n)\del^\tau=-d\del^\tau$ for any $1\in\tau\in F_d(\cK_n)$ using again~\eqref{eq:sign_simplices}.
Thus, $-d$ is the eigenvalue of $\cA_{d-1}(\cK_n)$ with multiplicity at least $\binom{n-1}d$.
Since $\binom{n-1}{d-1}+\binom{n-1}d=\binom nd=\dim\Om^{d-1}(\cK_n)$, the conclusion follows.
\end{eg}

\section{Expectation of moments}\label{sec:E_moments}
In this section, we briefly summarize the proof of Lemma~\ref{lem:KR17_Lem3.2} for the reader's convenience.
The concepts introduced here are also important for the subsequent sections.
In what follows in this section, we assume that $\lim_{n\to\infty}np(1-p)=\infty$.
Recall that
\begin{align*}
H_n
&\coloneqq\frac1{\sqrt{np(1-p)}}(A_{d-1}(Y^d_{n,p})-\E[A_{d-1}(Y^d_{n,p})])\\
&=\frac1{\sqrt{np(1-p)}}(A_{d-1}(Y^d_{n,p})-pA_{d-1}(\cK_n)]),
\end{align*}
and the conclusion of Lemma~\ref{lem:KR17_Lem3.2} is that for any $k\in\Z_{\ge0}$,
\[
\lim_{n\to\infty}\E\langle L_{H_n},x^k\rangle=\langle\nu_d,x^k\rangle=\begin{cases}
0			&\text{if $k$ is odd,}\\
d^{k/2}\cC_{k/2}	&\text{if $k$ is even,}
\end{cases}
\]
where $\cC_k\coloneqq\frac1{k+1}\binom{2k}k$ is the $k$th Catalan number for $k\ge1$, and $\cC_0\coloneqq1$.
When $k\le1$, the conclusion is trivial because $\langle L_{H_n},1\rangle=1$ and $\langle L_{H_n},x\rangle=0$ almost surely.
Hence, let $k\ge2$ in what follows in this section.
The starting point of the proof is the following expression
\begin{align}\label{eq:starting_point}
\langle L_{H_n},x^k\rangle
&=\frac1{\binom nd}\sum_{i=1}^{\binom nd}\lm_i[H_n]^k\nonumber\\
&=\frac1{\binom nd\{np(1-p)\}^{k/2}}\sum_{i=1}^{\binom nd}\lm_i[A_{d-1}(Y^d_{n,p})-\E A_{d-1}(Y^d_{n,p})]^k\nonumber\\
&=\frac1{\binom nd\{np(1-p)\}^{k/2}}\Tr\bigl((A_{d-1}(Y^d_{n,p})-\E A_{d-1}(Y^d_{n,p}))^k\bigr)\nonumber\\
&=\frac1{\binom nd\{np(1-p)\}^{k/2}}\sum_{\substack{\sg_1,\sg_2,\ldots,\sg_{k+1}\in F_{d-1}(\cK_n)\\\sg_1=\sg_{k+1}}}
\prod_{i=1}^k(A_{d-1}(Y^d_{n,p})-\E A_{d-1}(Y^d_{n,p}))_{\sg_i,\sg_{i+1}}.
\end{align}
Note that the summand in the last line is zero unless $\sg_i\cup\sg_{i+1}\in F_d(\cK_n)$ for all $i=1,2,\ldots,k$.
Hence, we can restrict the summation to such $\sg_i$'s.
The notion of $(n,d)$-words, introdued in Knowles--Rosenthal~\cite[Definition~3.4]{KR17}, is crucial for the further calculation of~\eqref{eq:starting_point}.
However, we herein introduce a slightly different definition of $(n,d)$-words (see Remark~\ref{rem:nd_word} for other definitions).
Our definition of $(n,d)$-words turns out to be critical in the subsequent sections, especially in the proofs of Lemmas~\ref{lem:cW_ks2}(3),~\ref{lem:cW_ks2-}, and~\ref{lem:cW_ks2+}.
\begin{df}[$(n,d)$-word]\label{df:nd_word}
A sequence $\widetilde\sg=(v_0,v_1,\ldots,v_{d-1})\in [n]^d$ consisting of distinct vertices is called an \textit{ordered $(d-1)$-simplex} in $\cK_n$.
Given an ordered $(d-1)$-simplex $\widetilde\sg=(v_0,v_1,\ldots,v_{d-1})$, we denote by $\sg$ its unordered simplex $\{v_0,v_1,\ldots,v_{d-1}\}$.
For $k\ge1$, we call a finite sequence $w=\widetilde{\sg_1}\sg_2\cdots\sg_k$ of an ordered $(d-1)$-simplex $\widetilde{\sg_1}$ and $\sg_2,\ldots,\sg_k\in F_{d-1}(\cK_n)$ an \textit{$(n,d)$-word} if $\sg_i\cup\sg_{i+1}\in F_d(\cK_n)$ for all $i=1,2,\ldots,k-1$.
Here, $k$ is called the \textit{length} of $w=\widetilde{\sg_1}\sg_2\cdots\sg_k$.
An $(n,d)$-word $w=\widetilde{\sg_1}\sg_2\cdots\sg_k$ is said to be \textit{closed} if $\sg_1=\sg_k$.
Two $(n,d)$-words $w=\widetilde{\sg_1}\sg_2\cdots\sg_k$ and $x=\widetilde{\rho_1}\rho_2\cdots\rho_k$ are said to be \textit{equivalent} if there exists a bijection $\pi$ on $[n]$ such that $\pi(w)=x$, i.e.,
\begin{itemize}
\item[(i)] the image $\pi(\sg_i)$ of $\sg_i\subset[n]$ by $\pi$ coincides with $\rho_i$ for each $i=1,2,\ldots,k$; and
\item[(ii)] if $\widetilde{\sg_1}=(v_0,v_1,\ldots,v_{d-1})$, then $\widetilde{\rho_1}=(\pi(v_0),\pi(v_1),\ldots,\pi(v_{d-1}))$.
(In other words, $\pi$ preserves the ordering on the initial ordered simplex.)
\end{itemize}
\end{df}
\begin{rem}\label{rem:nd_word}
There are a few variations of the definitions of $(n,d)$-words and equivalence relation among them.
In~\cite{KR17,LR22}, a $(n,d)$-word is defined in terms of positively oriented $(d-1)$-simplices in $\cK_n$.
However, by identifying each positively oriented simplex with the corresponding unoriented simplex appropriately, we explain it without introducing the notion of orientations on simplices.
In~\cite{KR17,LR22}, a $(n,d)$-word is defined as a finite sequence $w=\sg_1\sg_2\cdots\sg_k$ of elements in $F_{d-1}(\cK_n)$ with $\sg_i\cup\sg_{i+1}\in F_d(\cK_n)$ for all $i=1,2,\ldots,k-1$.
Two $(n,d)$-words $w=\sg_1\sg_2\cdots\sg_k$ and $x=\rho_1\rho_2\cdots\rho_k$ are said to be equivalent if there exists a bijection $\pi$ on $[n]$ such that
\begin{itemize}
\item[(i)] the image $\pi(\sg_i)$ coincides with $\rho_i$ for each $i=1,2,\ldots,k$; and
\item[(ii)] for each $i=1,2,\ldots,k$, if $\sg_i=\{v_0,v_1,\ldots,v_{d-1}\}$ with $v_0<v_1<\cdots<v_{d-1}$ and $\rho_i=\{u_0,u_1,\ldots,u_{d-1}\}$ with $u_0<u_1<\cdots<u_{d-1}$, then the two tuples $(\pi(v_0),\pi(v_1),\ldots,\pi(v_{d-1}))$ and $(u_0,u_1,\ldots,u_{d-1})$ differ by an even permutation.
(In other words, $\pi$ preserves the orientations on $\sg_i$'s.)
\end{itemize}
Although the definition of $(n,d)$-words in~\cite{AKS23} is the same as above, the equivalence relation among them are different.
In~\cite{AKS23}, the Condition~(ii) is replaced with the following condition.
\begin{itemize}
\item[(ii)'] if $\sg_1=\{v_0,v_1,\ldots,v_{d-1}\}$ with $v_0<v_1<\cdots<v_{d-1}$, then $\pi(v_0)<\pi(v_1)<\cdots<\pi(v_{d-1})$.
(In other words, $\pi$ preserves the canonical ordering on the initial simplex.)
\end{itemize}
\end{rem}
We also use the notion of the induced ordering of a given $(n,d)$-word later, which will be used in Remark~\ref{rem:KR17_Lem3.11} and the proofs of Lemmas~\ref{lem:cW_ks2-} and~\ref{lem:cW_ks2+}.
\begin{df}[Induced ordering]\label{df:induced_ordering}
Given an $(n,d)$-word $w=\widetilde{\sg_1}\sg_2\cdots\sg_k$, we construct ordered $(d-1)$-simplices $\widetilde{\sg_2},\widetilde{\sg_3},\ldots,\widetilde{\sg_k}$ inductively as follows.
Suppose that we obtain ordered $(d-1)$-simplices $\widetilde{\sg_1},\widetilde{\sg_2},\ldots,\widetilde{\sg_i}$ for some $i=1,2,\ldots,k-1$ and also that $\widetilde{\sg_i}=(v_0,v_1,\ldots,v_{d-1})$ and $\sg_{i+1}=\{v\}\cup\{v_0,v_1,\ldots,\hat v_j,\ldots,v_{d-1}\}$ for some $j=0,1,\ldots,d-1$.
We then define $\widetilde{\sg_{i+1}}\coloneqq(v,v_0,v_1,\ldots,\hat v_j,\ldots,v_{d-1})$.
For any $\sg\in\{\sg_1,\sg_2,\ldots,\sg_k\}$, the \textit{induced ordering} on $\sg$ by $w$ is defined as $\widetilde{\sg_{i'}}$, where $i'\in\{1,2,\ldots,k\}$ is the first index such that $\sg_{i'}=\sg$.
\end{df}
Other important notions are the support complex and sign of a given $(n,d)$-word.
\begin{df}[Support complex of $(n,d)$-word]
Given an $(n,d)$-word $w=\widetilde{\sg_1}\sg_2\cdots\sg_k$, its \textit{vertex support} and \textit{$d$-simplex support} are defined as
\[
\Supp_0(w)\coloneqq\sg_1\cup\sg_2\cup\cdots\cup\sg_k\quad\text{and}\quad\Supp_d(w)\coloneqq\{\sg_i\cup\sg_{i+1}\mid i=1,2,\ldots,k-1\},
\]
respectively.
Furthermore, let $X_w$ be the simplicial complex generated by $\Supp_d(w)$, i.e., the smallest simplicial complex containing $\Supp_d(w)$.
We call $X_w$ the \textit{support complex} of $w$.
\end{df}
\begin{df}[Sign of $(n,d)$-word]
Let $w=\widetilde{\sg_1}\sg_2\cdots\sg_k$ be an $(n,d)$-word.
For each $\tau\in\Supp_d(w)$, we define the \textit{sign of $w$ at $\tau$} by
\[
\sgn(w,\tau)\coloneqq\prod_{\substack{i\in\{1,2,\ldots,k-1\}\\\sg_i\cup\sg_{i+1}=\tau}}\sgn(\sg_i,\sg_{i+1}).
\]
We also define the \textit{sign of $w$} by
\begin{equation}\label{eq:sgn_w}
\sgn(w)\coloneqq\prod_{\tau\in\Supp_d(w)}\sgn(w,\tau)=\prod_{i=1}^{k-1}\sgn(\sg_i,\sg_{i+1}).
\end{equation}
\end{df}
\begin{rem}
Note that whenever the $(n,d)$-word $w$ is closed, $\sgn(w)$ does not depend on the choice of the total order in $V(X)$ because a change of the ordering of the vertices in a simplex affects an even number of $\sgn(\sg_i,\sg_{i+1})$'s in~\eqref{eq:sgn_w}.
In particular, if $w$ and $x$ are equivalent closed $(n,d)$-words, then $\sgn(w)=\sgn(x)$.
\end{rem}
%
%
%
%

In terms of $(n,d)$-words,~\eqref{eq:starting_point} can be written as
\begin{equation}\label{eq:kth_moment}
\langle L_{H_n},x^k\rangle =\frac1{d!\binom nd\{np(1-p)\}^{k/2}}\sum_{w=\widetilde{\sg_1}\sg_2\cdots\sg_{k+1}}T_n(w),
\end{equation}
where $w=\widetilde{\sg_1}\sg_2\cdots\sg_{k+1}$ in the summation runs over all closed $(n,d)$-words of length $k+1$, and
\[
T_n(w)\coloneqq\prod_{i=1}^k(A_{d-1}(Y^d_{n,p})-\E A_{d-1}(Y^d_{n,p}))_{\sg_i,\sg_{i+1}}.
\]
From the independence of the appearance of $d$-simplices in $Y^d_{n,p}$,
\[
\bar T_n(w)\coloneqq\E T_n(w)=\sgn(w)\prod_{\tau\in\Supp_d(w)}\E\bigl[(\chi_p-p)^{N_w(\tau)}\bigr].
\]
Here, $\chi_p$ is a Bernoulli random variable with parameter $p$, and $N_w(\tau)$ denotes the number of times the $(n,d)$-word $w=\widetilde{\sg_1}\sg_2\cdots\sg_{k+1}$ traverses the $d$-simplex $\tau$:
\[
N_w(\tau)\coloneqq|\{i\in\{1,2,\ldots,k\}\mid\sg_i\cup\sg_{i+1}=\tau\}|.
\]
We note that $\bar T_n(w)=0$ unless $N_w(\tau)\ge2$ for all $\tau\in\Supp_d(w)$.
Furthermore, $\bar T_n(w)=\bar T_n(x)$ if $w$ and $x$ are equivalent $(n,d)$-words.
For $1\le s\le n$, let $\cW_{k,s}$ denote the set of all representatives for the equivalence classes of closed $(n,d)$-words $w$ of length $k+1$ such that $N_w(\tau)\ge2$ for all $\tau\in\Supp_d(w)$ and $|\Supp_0(w)|=s$.
From the above discussion, we obtain
\begin{equation}\label{eq:E_kth_moment}
\E\langle L_{H_n},x^k\rangle
=\frac1{d!\binom nd\{np(1-p)\}^{k/2}}\sum_{s=1}^n \sum_{w\in\cW_{k,s}}|[w]|\bar T_n(w).
\end{equation}
Here, $[w]$ denotes the equivalence class of the $(n,d)$-word $w$.
For a further calculation of~\eqref{eq:E_kth_moment}, we use the following lemmas.
%
\begin{lem}[cf.~{\cite[Claims~3.9 and~3.10]{KR17}}]\label{lem:cW_ks}
Let $k\ge2$ and $1\le s\le n$.
Then, the following statements hold.
\begin{enumerate}
\item If $\cW_{k,s}\neq\emptyset$, then $d+1\le s\le\lfloor k/2\rfloor+d$.
\item If $d+1\le s\le\lfloor k/2\rfloor+d$, then $|\cW_{k,s}|\le d!(\lfloor k/2\rfloor+1)^{dk}$.
\item For every $w\in\cW_{k,s}$, there exist exactly $n(n-1)\cdots(n-s+1)$ number of $(n,d)$-words that are equivalent to $w$.
\end{enumerate}
\end{lem}
\begin{proof}
(1)~Since $k\ge2$, the first inequality is trivial.
For the second inequality, let $w\in\cW_{k,s}$.
From the definition of $\cW_{k,s}$,
\[
k=\sum_{\tau\in\Supp_d(w)}N_w(\tau)\ge2|\Supp_d(w)|.
\]
Therefore, $s=|\Supp_0(w)|\le|\Supp_d(w)|+d\le\lfloor k/2\rfloor+d$.
Here, the first inequality follows from Lemma~\ref{lem:f0-fd} since $X_w$ is a strongly connected pure $d$-dimensional simplicial complex.

(2)~We may assume that all the $(n,d)$-words $w$ in $\cW_{k,s}$ are supported on $\{1,2,\ldots,s\}$, i.e., $\Supp_0(w)=\{1,2,\ldots,s\}$.
Since there are $\binom sd$ possible $(d-1)$-simplices whose vertices are in $\{1,2,\ldots,s\}$, the number of closed $(n,d)$-words $w$ of length $k+1$ such that $\Supp_0(w)=\{1,2,\ldots,s\}$ is bounded above by
\begin{align*}
d!\binom sd^k
&\le d!\biggl(\frac{s(s-1)\cdots(s-d+1)}{d!}\biggr)^k\\
&\le d!\biggl(\frac{\lfloor k/2\rfloor+d}d\frac{\lfloor k/2\rfloor+d-1}{d-1}\cdots\frac{(\lfloor k/2\rfloor+1)}1\biggr)^k
\le d!(\lfloor k/2\rfloor+1)^{dk}.
\end{align*}

(3)~Let $w=\widetilde{\sg_1}\sg_2\cdots\sg_{k+1}\in\cW_{k,s}$.
It suffices to construct a bijective map $\Phi$ from the equivalence class $[w]$ to the set of all tuples of the forms $(v_1,v_2,\ldots,v_s)$ consisting of distinct vertices $v_1,v_2,\ldots,v_s\in[n]$.
Such bijective map $\Phi$ is given in the following way.
Let $x=\widetilde{\rho_1}\rho_2\cdots\rho_{k+1}\in[w]$ be an $(n,d)$-word with $\widetilde{\rho_1}=(v_1,v_2,\ldots,v_d)$.
We can further take distinct vertices $v_{d+1},v_{d+2},\ldots,v_s$ satisfying that for any $1\le i\le k+1$, the truncated $(n,d)$-word $\widetilde{\rho_1}\rho_2\cdots\rho_i$ is supported on $\{v_1,v_2,\ldots,v_r\}$ for some $d\le r\le s$.
Noting that $(v_1,v_2,\ldots,v_s)$ is uniquely determined by $x$, we define $\Phi(x)\coloneqq(v_1,v_2,\ldots,v_s)$.
We can easily verify that the map $\Phi$ is bijective.
\end{proof}
The following is a slightly simpler version of~\cite[Lemma~3.11]{KR17}.
\begin{lem}[{cf.~\cite[Lemma~3.11]{KR17}}]\label{lem:KR17_Lem3.11}
Let $k\ge2$ be even, and let $w\in\cW_{k,k/2+d}$ be fixed.
Then, the following hold.
\begin{enumerate}
\item $N_w(\tau)=2$ for every $\tau\in\Supp_d(w)$.
In particular, $|\Supp_d(w)|=k/2$.
\item $\sgn(w,\tau)=1$ for every $\tau\in\Supp_d(w)$.
\item $|\cW_{k,k/2+d}|=d^{k/2}\cC_{k/2}$.
\end{enumerate}
\end{lem}
\begin{rem}\label{rem:KR17_Lem3.11}
A key for the proof of Lemma\ref{lem:KR17_Lem3.11}(3) is to construct a bijection from $\cW_{k,k/2+d}$ to the set of all rooted planar trees with $k/2$ edges, to each of which an integer from $\{1,2,\cdots,d\}$ is assigned.
Once we obtain such a bijection, the conclusion follows immediately from the fact that the number of rooted planar trees with $k/2$ edges is given by the $(k/2)$th Catalan number $\cC_{k/2}$ and that the number of possible labeling on edges is $d^{k/2}$.
Although we adopt slightly different definitions of $(n,d)$-words and the equivalence relation among them from those in~\cite{KR17}, the construction of the rooted planar trees with $k/2$ edges from a given $(n,d)$-word $w=\widetilde{\sg_1}\sg_2\cdots\sg_{k+1}\in\cW_{k,k/2+d}$ is the same as that in the proof of~\cite[Lemma~3.11]{KR17}, where the vertex set and edge set of each rooted planar tree are identified as $\{\sg_i\mid i=1,2,\ldots,k\}$ and $\{\{\sg_i,\sg_{i+1}\}\mid i=1,2,\ldots,k\}$, respectively (see~\cite[Lemma~3.11]{KR17} for more details).
However, one should be careful in labeling to each edge $\{\sg,\sg'\}$ in the rooted planar tree.
We assign an integer to an edge $\{\sg,\sg'\}$ in the rooted planar tree according to what number of vertex in the induced ordering is removed when the word $w=\widetilde{\sg_1}\sg_2\cdots\sg_{k+1}$ traverses $\{\sg,\sg'\}$ for the first time.
\end{rem}

Combining Lemmas~\ref{lem:cW_ks} and~\ref{lem:KR17_Lem3.11} with~\eqref{eq:E_kth_moment}, the proof of Lemma~\ref{lem:KR17_Lem3.2} proceeds as follows.
\begin{proof}[Proof of Lemma~\ref{lem:KR17_Lem3.2}]
From~\eqref{eq:E_kth_moment} and Lemma~\ref{lem:cW_ks}(1)(3), we have
\begin{equation}\label{eq:KR17_Lem3.2_1}
\E\langle L_{H_n},x^k\rangle
=\sum_{s=d+1}^{\lfloor k/2\rfloor+d}\sum_{w\in\cW_{k,s}}\frac{(n-d)(n-d-1)\cdots(n-s+1)}{\{np(1-p)\}^{k/2}}\bar T_n(w).
\end{equation}
Now, for any $d+1\le s\le\lfloor k/2\rfloor+d$ and $w\in\cW_{k,s}$,
\begin{align}\label{eq:KR17_Lem3.2_2}
|\bar T_n(w)|=\Biggl|\prod_{\tau\in\Supp_d(w)}\E\bigl[(\chi_p-p)^{N_w(\tau)}\bigr]\Biggr|
&=\prod_{\tau\in\Supp_d(w)}\bigl|p(1-p)^{N_w(\tau)}+(1-p)(-p)^{N_w(\tau)}\bigr|\nonumber\\
&\le\prod_{\tau\in\Supp_d(w)}\bigl[p(1-p)\bigl\{(1-p)^{N_w(\tau)-1}+p^{N_w(\tau)-1}\bigr\}\bigr]\nonumber\\
&\le\{p(1-p)\}^{|\Supp_d(w)|}\nonumber\\
&\le\{p(1-p)\}^{s-d}.
\end{align}
For the third line, we note that $(1-x)^m+x^m\le1$ for any $x\in[0,1]$ and $m\in\N$.
For the last line, we used Lemma~\ref{lem:f0-fd}.
Therefore, it follows from Lemma~\ref{lem:cW_ks}(2) that for any $d+1\le s<k/2+d$,
\begin{align*}
\Biggl|\sum_{w\in\cW_{k,s}}\frac{(n-d)(n-d-1)\cdots(n-s+1)}{\{np(1-p)\}^{k/2}}\bar T_n(w)\Biggr| \le\frac{d!(\lfloor k/2\rfloor+1)^{dk}}{\{np(1-p)\}^{k/2+d-s}} \le\frac{d!(\lfloor k/2\rfloor+1)^{dk}}{\sqrt{np(1-p)}}.
\end{align*}
Thus, for odd $k$, $\lim_{n\to\infty}\E\langle L_{H_n},x^k\rangle=0$.
Furthermore, when $k$ is even, we have
\[
\E\langle L_{H_n},x^k\rangle
=\sum_{w\in\cW_{k,k/2+d}}\frac{(n-d)(n-d-1)\cdots(n-(k/2+d)+1)}{\{np(1-p)\}^{k/2}}\bar T_n(w)+o_{d,k}(1).
\]
Here, $o_{d,k}(1)$ is a function of $n$, depending on $d$ and $k$, that converges to zero as $n\to\infty$.
Finally, Lemma~\ref{lem:KR17_Lem3.11} yields
\begin{align*}
&\sum_{w\in\cW_{k,k/2+d}}\frac{(n-d)(n-d-1)\cdots(n-(k/2+d)+1)}{\{np(1-p)\}^{k/2}}\bar T_n(w)\\
&=d^{k/2}\cC_{k/2}\frac{(n-d)(n-d-1)\cdots(n-(k/2+d)+1)}{\{np(1-p)\}^{k/2}}\E\bigl[(\chi_p-p)^2\bigr]^{k/2}\\
&=d^{k/2}\cC_{k/2}\frac{(n-d)(n-d-1)\cdots(n-(k/2+d)+1)}{n^{k/2}}\\
&\xrightarrow[n\to\infty]{}d^{k/2}\cC_{k/2},
\end{align*}
which completes the proof.
\end{proof}

\section{Covariance of moments}\label{sec:cov_moments}
Our aim in this section is to prove Theorem~\ref{thm:multi_CLT}(1).
Recall that $p_\infty\coloneqq\lim_{n\to\infty}p\in[0,1]$, and $\sg(k,l)\in[0,\infty)$ is given as follows: if $k+l$ is even, then
\begin{align*}
\sg(k,l)
&\coloneqq\1_{\{k,l\colon\text{even}\}}(d+1)!\{(k/2)d^{k/2}\cC_{k/2}\}\{(l/2)d^{l/2}\cC_{l/2}\}(2p_\infty-1)^2\\
&\qad+\sum_{r=3}^{k\wedge l}\frac2r\Biggl(d!k\sum_{\substack{k_1,k_2,\ldots,k_r\in2\Z_{\ge0}\\k_1+\cdots+k_r=k-r}}\prod_{q=1}^rd^{k_q/2}\cC_{k_q/2}\Biggr)\Biggl(d!l\sum_{\substack{l_1,l_2,\ldots,l_r\in2\Z_{\ge0}\\l_1+\cdots+l_r=l-r}}\prod_{q=1}^rd^{l_q/2}\cC_{l_q/2}\Biggr)p_\infty(1-p_\infty),
\end{align*}
and $\sg(k,l)\coloneqq0$ otherwise.
In what follows in this section, we also assume that $\lim_{n\to\infty}np(1-p)=\infty$.
Our goal here is to prove that for any $k,l\in\Z_{\ge0}$,
\[
\lim_{n\to\infty}n^d\{np(1-p)\}\cdot\cov(\langle L_{H_n},x^k\rangle,\langle L_{H_n},x^l\rangle)
=\sg(k,l).
\]

\subsection{Proof of Theorem~\ref{thm:multi_CLT}(1)}
When $k\wedge l\le1$, the conclusion of Theorem~\ref{thm:multi_CLT}(1) is trivial because $\langle L_{H_n},1\rangle=1$ and $\langle L_{H_n},x\rangle=0$ almost surely.
Hence, let $k,l\ge2$ in what follows in this section.
We first recall that
\[
\langle L_{H_n},x^k\rangle =\frac1{d!\binom nd\{np(1-p)\}^{k/2}}\sum_{w=\widetilde{\sg_1}\sg_2\cdots\sg_{k+1}}T_n(w)
\]
as seen in~\eqref{eq:kth_moment}.
Here, $w=\widetilde{\sg_1}\sg_2\cdots\sg_{k+1}$ in the summation runs over all closed $(n,d)$-words of length $k+1$, and
\[
T_n(w)=\prod_{i=1}^k(A_{d-1}(Y^d_{n,p})-\E A_{d-1}(Y^d_{n,p}))_{\sg_i,\sg_{i+1}}.
\]
Therefore,
\begin{align}\label{eq:V_kth_moment_1}
&\cov(\langle L_{H_n},x^k\rangle,\langle L_{H_n},x^l\rangle)\nonumber\\
&=\frac1{\bigl\{d!\binom nd\bigr\}^2\{np(1-p)\}^{(k+l)/2}}\sum_{\substack{w^1=\widetilde{\sg_1^1}\sg_2^1\cdots\sg_{k+1}^1\\w^2=\widetilde{\sg_1^2}\sg_2^2\cdots\sg_{l+1}^2}}\E\bigl[\bigl(T_n(w^1)-\bar T_n(w^1)\bigr)\bigl(T_n(w^2)-\bar T_n(w^2)\bigr)\bigr]
\end{align}
Here, $w^1=\widetilde{\sg_1^1}\sg_2^1\cdots\sg_{k+1}^1$ and $w^2=\widetilde{\sg_1^1}\sg_2^2\cdots\sg_{l+1}^2$ in the summation run over all closed $(n,d)$-words of length $k+1$ and $l+1$, respectively.
The role of $(n,d)$-words in Section~\ref{sec:E_moments} is now played by $(n,d)$-sentences.
\begin{df}[$(n,d)$-sentence]
For $h\ge1$, we call a finite sequence $a=(w^1,w^2,\ldots,w^h)$ of $(n,d)$-words an \textit{$(n,d)$-sentence}.
Two $(n,d)$-sentences $a=(w^1,w^2,\ldots,w^h)$ and $b=(x^1,x^2,\ldots,x^h)$ are said to be \textit{equivalent} if there exists a permutation $\pi$ on $[n]$ such that $\pi(w^j)=x^j$ for $j\in[h]\coloneqq\{1,2,\ldots,h\}$ in the sense of Definition~\ref{df:nd_word}.
\end{df}
\begin{df}[Support complex of $(n,d)$-sentence]
Given an $(n,d)$-sentence $a=(w_1,w_2\ldots,w_h)$, its \textit{vertex support} and \textit{$d$-simplex support} are defined as
\[
\Supp_0(a)\coloneqq\bigcup_{j=1}^h\Supp_0(w^j)\quad\text{and}\quad\Supp_d(a)\coloneqq\bigcup_{j=1}^h\Supp_d(w^j),
\]
respectively.
Furthermore, we call $X_a\coloneqq\bigcup_{j=1}^h X_{w^j}$ the \textit{support complex} of $a$.
\end{df}
\begin{df}[Sign of $(n,d)$-sentence]
Given an $(n,d)$-sentence $a=(w_1,w_2\ldots,w_h)$, we define the \textit{sign of $a$} by
\[
\sgn(a)\coloneqq\prod_{j=1}^h\sgn(w^j).
\]
\end{df}
%
%
For $(n,d)$-sentence $a=(w^1,w^2)$ consisting of two closed $(n,d)$-words $w^1=\widetilde{\sg_1^1}\sg_2^1\cdots\sg_{k+1}^1$ and $w^2=\widetilde{\sg_1^1}\sg_2^2\cdots\sg_{l+1}^2$ of length $k+1$ and $l+1$, respectively, we set
\begin{align}\label{eq:barTa}
\bar T_n(a)
&\coloneqq\E\bigl[\bigl(T_n(w^1)-\bar T_n(w^1)\bigr)\bigl(T_n(w^2)-\bar T_n(w^2)\bigr)\bigr]\nonumber\\
&=\E[T_n(w^1)T_n(w^2)]-\bar T_n(w^1)\bar T_n(w^2)\nonumber\\
&=\E\Biggl[\Biggl(\prod_{i=1}^k(A_{d-1}(Y^d_{n,p})-\E A_{d-1}(Y^d_{n,p}))_{\sg_i^1,\sg_{i+1}^1}\Biggr)\Biggl(\prod_{i=1}^l(A_{d-1}(Y^d_{n,p})-\E A_{d-1}(Y^d_{n,p}))_{\sg_i^2,\sg_{i+1}^2}\Biggr)\Biggr]\nonumber\\
&\qad-\E\Biggl[\prod_{i=1}^k(A_{d-1}(Y^d_{n,p})-\E A_{d-1}(Y^d_{n,p}))_{\sg_i^1,\sg_{i+1}^1}\Biggr]\E\Biggl[\prod_{i=1}^k(A_{d-1}(Y^d_{n,p})-\E A_{d-1}(Y^d_{n,p}))_{\sg_i^2,\sg_{i+1}^2}\Biggr]\nonumber\\
&=\sgn(a)\prod_{\tau\in\Supp_d(a)}\E\bigl[(\chi_p-p)^{N_a(\tau)}\bigr]\nonumber\\
&\qad-\sgn(a)\Biggl(\prod_{\tau\in\Supp_d(w^1)}\E\bigl[(\chi_p-p)^{N_{w^1}(\tau)}\bigr]\Biggr)\Biggl(\prod_{\tau\in\Supp_d(w^1)}\E\bigl[(\chi_p-p)^{N_{w^2}(\tau)}\bigr]\Biggr).
\end{align}
Here, $N_a(\tau)\coloneqq N_{w^1}(\tau)+N_{w^2}(\tau)$.
Note that $\bar T_n(a)=0$ unless $\Supp_d(w^1)\cap\Supp_d(w^2)\neq\emptyset$.
Also, $\bar T_n(a)=0$ unless $N_a(\tau)\ge2$ for all $\tau\in\Supp_d(a)$.
Furthermore, $\bar T_n(a)=\bar T_n(b)$ if $a$ and $b$ are equivalent $(n,d)$-sentences.
For $1\le s\le n$, let $\cW_{k,l,s}^{(2)}$ denote the set of all representatives for the equivalence classes of sentences $a$ consisting of two closed $(n,d)$-words $w^1$ and $w^2$ of length $k+1$ and $l+1$, respectively, such that $\Supp_d(w^1)\cap\Supp_d(w^2)\neq\emptyset$, $N_a(\tau)\ge2$ for all $\tau\in\Supp_d(a)$, and $|\Supp_0(a)|=s$.
Then, from~\eqref{eq:V_kth_moment_1}, we obtain
\begin{equation}\label{eq:V_kth_moment_2}
\cov(\langle L_{H_n},x^k\rangle,\langle L_{H_n},x^l\rangle) =\frac1{\bigl\{d!\binom nd\bigr\}^2\{np(1-p)\}^{(k+l)/2}}\sum_{s=1}^n \sum_{a=(w^1,w^2)\in\cW_{k,l,s}^{(2)}}|[a]|\bar T_n(a).
\end{equation}
Here, $[a]$ denotes the equivalence class of the $(n,d)$-sentence $a$.
For a further calculation of~\eqref{eq:V_kth_moment_2}, we use the following lemma.
\begin{lem}\label{lem:cW_ks2}
Let $k,l\ge2$ and $1\le s\le n$.
Then, the following statements hold.
\begin{enumerate}
\item If $\cW_{k,l,s}^{(2)}\neq\emptyset$, then $d+1\le s\le\lfloor(k+l)/2\rfloor+d-1$.
\item If $d+1\le s\le\lfloor(k+l)/2\rfloor+d-1$, then $|\cW_{k,l,s}^{(2)}|\le(d!)^2\lfloor(k+l)/2\rfloor^{d(k+l)}$.
\item For any $a\in\cW_{k,l,s}^{(2)}$, there exist exactly $n(n-1)\cdots(n-s+1)$ number of $(n,d)$-sentences that are equivalent to $a$.
\end{enumerate}
\end{lem}
\begin{proof}
(1)~Since $k,l\ge2$, the first inequality is trivial.
For the second inequality, let $a=(w^1,w^2)\in\cW_{k,l,s}^{(2)}$.
From the definition of $\cW_{k,l,s}^{(2)}$, \begin{equation}\label{eq:cW_ks2_1}
k+l=\sum_{\tau\in\Supp_d(a)}N_a(\tau)\ge2|\Supp_d(a)|.
\end{equation}
Since $X_a$ is a strongly connected pure $d$-dimensional simplicial complex, we have $|\Supp_0(a)|\le|\Supp_d(a)|+d$ from Lemma~\ref{lem:f0-fd}.
Therefore, we obtain
\begin{equation}\label{eq:cW_ks2_2}
s=|\Supp_0(a)|\le|\Supp_d(a)|+d\le\lfloor(k+l)/2\rfloor+d.
\end{equation}
We will exclude the possibility that $|\Supp_0(a)|=\lfloor(k+l)/2\rfloor+d$ by contradiction.
Assume that $|\Supp_0(a)|=\lfloor(k+l)/2\rfloor+d$.
Then,~\eqref{eq:cW_ks2_2} yields $|\Supp_d(a)|=\lfloor(k+l)/2\rfloor$.
In particular, $|\Supp_0(a)|=|\Supp_d(a)|+d$ holds, which indicates that $X_a$ is a $d$-tree.
Therefore, for $j=1,2$ and $\tau\in\Supp_d(w^j)$, we have $N_{w^j}(\tau)\ge2$ since $w^j$ is a closed $(n,d)$-word.
On the other hand, it follows from~\eqref{eq:cW_ks2_1} and $|\Supp_d(a)|=\lfloor(k+l)/2\rfloor$ that $N_a(\tau)\le3$ for every $\tau\in\Supp_d(a)$.
Consequently, $\Supp_d(w^1)\cap\Supp_d(w^2)=\emptyset$, contradicting the definition of $\cW_{k,l,s}^{(2)}$.

(2)~We may assume that all the $(n,d)$-sentences $a$ in $\cW_{k,l,s}^{(2)}$ are supported on $\{1,2,\ldots,s\}$, i.e., $\Supp_0(a)=\{1,2,\ldots,s\}$.
Since there are $\binom sd$ possible $(d-1)$-simplices whose vertices are in $\{1,2,\ldots,s\}$, the number of closed $(n,d)$-sentences $a$ consisting of two closed $(n,d)$-words $w^1$ and $w^2$ of length $k+1$ and $l+1$, respectively, such that $\Supp_0(w^j)\subset\{1,2,\ldots,s\}$ for $j=1,2$ is bounded above by
\begin{align*}
(d!)^2\binom sd^{k+l}
&\le(d!)^2\biggl(\frac{s(s-1)\cdots(s-d+1)}{d!}\biggr)^{k+l}\\
&\le(d!)^2\biggl(\frac{\lfloor(k+l)/2\rfloor+d-1}d\frac{\lfloor(k+l)/2\rfloor+d-2}{d-1}\cdots\frac{\lfloor(k+l)/2\rfloor}1\biggr)^{k+l}\\
&\le(d!)^2\lfloor(k+l)/2\rfloor^{d(k+l)}.
\end{align*}

(3)~Let $a=(w^1,w^2)\in\cW_{k,l,s}^{(2)}$.
It suffices to construct a bijective map $\Phi$ from $[a]$ to the set of all tuples of the forms $(v_1,v_2,\ldots,v_s)$ consisting of distinct vertices $v_1,v_2,\ldots,v_s\in[n]$.
Such bijective map $\Phi$ is given in the following way.
Let $b=(x^1,x^2)\in[a]$ be an $(n,d)$-sentence consisting of two closed $(n,d)$-words $x^1=\widetilde{\rho^1_1}\rho^2_1\cdots\rho^1_{k+1}$ and $x^2=\widetilde{\rho^2_1}\rho^2_2\cdots\rho^2_{l+1}$ of length $k+1$ and $l+1$, respectively.
Set $\widetilde{\rho^1_1}=(v_1,v_2,\ldots,v_d)$.
We can further take distinct vertices $v_{d+1},v_{d+2},\ldots,v_s$ satisfying that for any $1\le i\le k+1$, the truncated $(n,d)$-word $\widetilde{\rho^1_1}\rho^1_2\cdots\rho^1_i$ is supported on $\{v_1,v_2,\ldots,v_r\}$ for some $d\le r\le|\Supp_0(x^1)|$, and that for any $1\le i'\le l+1$, the truncated $(n,d)$-sentence $(x^1,\widetilde{\rho^2_1}\rho^2_2\cdots\rho^2_{i'})$ is supported on $\{v_1,v_2,\ldots,v_{r'}\}$ for some $|\Supp_0(x^1)|\le r'\le s$.
We may additionally assume that the vertices in $\rho^2_1\setminus\Supp_0(x^1)$ are sorted in $v_{d+1},v_{d+2},\ldots,v_s$ consistent with the ordering of $\widetilde{\rho^2_1}$.
Then, noting that $(v_1,v_2,\ldots,v_s)$ is uniquely determined by $b$, we define $\Phi(b)\coloneqq(v_1,v_2,\ldots,v_s)$.
We can easily verify that the map $\Phi$ is bijective.
\end{proof}
Combining Lemma~\ref{lem:cW_ks2}(1)(3) with~\eqref{eq:V_kth_moment_2}, we have
\[
\cov(\langle L_{H_n},x^k\rangle,\langle L_{H_n},x^l\rangle)
=\sum_{s=d+1}^{\lfloor(k+l)/2\rfloor+d-1}\sum_{a=(w^1,w^2)\in\cW_{k,l,s}^{(2)}}\frac{(n-d)(n-d-1)\cdots(n-s+1)}{d!\binom nd\{np(1-p)\}^{(k+l)/2}}\bar T_n(a).
\]
Since for any $d+1\le s\le\lfloor(k+l)/2\rfloor+d-1$ and $a=(w^1,w^2)\in\cW_{k,l,s}^{(2)}$,
\begin{align}\label{eq:V_kth_moment_3}
\Biggl|\prod_{\tau\in\Supp_d(a)}\E\bigl[(\chi_p-p)^{N_a(\tau)}\bigr]\Biggr|
&=\prod_{\tau\in\Supp_d(a)}\bigl|p(1-p)^{N_a(\tau)}+(1-p)(-p)^{N_a(\tau)}\bigr|\nonumber\\
&\le\prod_{\tau\in\Supp_d(a)}p(1-p)\bigl\{(1-p)^{N_a(\tau)-1}+p^{N_a(\tau)-1}\bigr\}\nonumber\\
&\le\{p(1-p)\}^{|\Supp_d(a)|}
\end{align}
and similarly
\begin{align}\label{eq:V_kth_moment_4}
\Biggl|\prod_{\tau\in\Supp_d(w^1)}\E\bigl[(\chi_p-p)^{N_{w^1}(\tau)}\bigr]\Biggr|\Biggl|\prod_{\tau\in\Supp_d(w^2)}\E\bigl[(\chi_p-p)^{N_{w^2}(\tau)}\bigr]\Biggr|
&\le\{p(1-p)\}^{|\Supp_d(w^1)|+|\Supp_d(w^2)|}\nonumber\\
&\le\{p(1-p)\}^{|\Supp_d(a)|},
\end{align}
we have $|\bar T_n(a)|\le2\{p(1-p)\}^{|\Supp_d(a)|}\le2\{p(1-p)\}^{s-d}$.
Here, we also used Lemma~\ref{lem:f0-fd} for the second inequality.
Therefore, it follows from Lemma~\ref{lem:cW_ks2}(2) that for any $d+1\le s<(k+l)/2+d-1$,
\begin{align*}
\Biggl|\sum_{a=(w^1,w^2)\in\cW_{k,l,s}^{(2)}}\frac{(n-d)(n-d-1)\cdots(n-s+1)}{d!\binom nd\{np(1-p)\}^{(k+l)/2}}\bar T_n(a)\Biggr|
&\le\frac{2(d!)^2\lfloor(k+l)/2\rfloor^{d(k+l)}}{d!\binom nd\{np(1-p)\}^{(k+l)/2+d-s}}\\
&\le\frac{2d!\lfloor(k+l)/2\rfloor^{d(k+l)}}{\binom nd\{np(1-p)\}^{3/2}}.
\end{align*}
Thus, when $k+l$ is odd,
\begin{equation}\label{eq:V_kth_moment_5}
\lim_{n\to\infty}n^d\{np(1-p)\}\cov(\langle L_{H_n},x^k\rangle,\langle L_{H_n},x^l\rangle)=0.
\end{equation}
Hence, we assume that $k+l$ is even in what follows in this section.
Then, we have
\begin{align}\label{eq:V_kth_moment_6}
&n^d\{np(1-p)\}\cdot\cov(\langle L_{H_n},x^k\rangle,\langle L_{H_n},x^l\rangle)\nonumber\\ &=\sum_{a=(w^1,w^2)\in\cW_{k,l,(k+l)/2+d-1}^{(2)}}\frac{n^d(n-d)(n-d-1)\cdots(n-\{(k+l)/2+d-1\}+1)}{d!\binom nd\{np(1-p)\}^{(k+l)/2-1}}\bar T_n(a)
+o_{d,k,l}(1)\nonumber\\
&=(1+o_{d,k,l}(1))\sum_{a=(w^1,w^2)\in\cW_{k,l,(k+l)/2+d-1}^{(2)}}\frac{\bar T_n(a)}{\{p(1-p)\}^{(k+l)/2-1}}
+o_{d,k,l}(1).
\end{align}
Here, $o_{d,k,l}(1)$ is a function of $n$, depending on $d$, $k$, and $l$, that converges to zero as $n\to\infty$.
Noting that $|\Supp_d(a)|=(k+l)/2-1\text{ or }(k+l)/2$ for every $a\in\cW_{k,l,(k+l)/2+d-1}^{(2)}$ by~\eqref{eq:cW_ks2_2}, we divide $\cW_{k,l,(k+l)/2+d-1}^{(2)}$ according to $|\Supp_d(a)|$:
\begin{align*}
\cW_{k,l,(k+l)/2+d-1}^{(2),-}&\coloneqq\{a\in\cW_{k,l,(k+l)/2+d-1}^{(2)}\mid|\Supp_d(a)|=(k+l)/2-1\}
\shortintertext{and}
\cW_{k,l,(k+l)/2+d-1}^{(2),+}&\coloneqq\{a\in\cW_{k,l,(k+l)/2+d-1}^{(2)}\mid|\Supp_d(a)|=(k+l)/2\}.
\end{align*}
The following lemmas are crucial for the proof of Theorem~\ref{thm:multi_CLT}(1).
\begin{lem}\label{lem:cW_ks2-}
Let $k,l\ge2$ such that $k+l$ is even.
Then, it holds that
\[
|\cW_{k,l,(k+l)/2+d-1}^{(2),-}|=\begin{cases}
(d+1)!\{(k/2)d^{k/2}\cC_{k/2}\}\{(l/2)d^{l/2}\cC_{l/2}\}           &\text{if both $k$ and $l$ are even,}\\
0           &\text{otherwise.}
\end{cases}
\]
Furthermore, $\bar T_n(a)=\{p(1-p)\}^{(k+l)/2-1}(2p-1)^2$ for every $a=(w^1,w^2)\in\cW_{k,l,(k+l)/2+d-1}^{(2),-}$.
\end{lem}
\begin{lem}\label{lem:cW_ks2+}
Let $k,l\ge2$ such that $k+l$ is even.
Then, it holds that
\[
|\cW_{k,l,(k+l)/2+d-1}^{(2),+}|
=\sum_{r=3}^{k\wedge l}\frac2r\Biggl(d!k\sum_{\substack{k_1,k_2,\ldots,k_r\in2\Z_{\ge0}\\k_1+\cdots+k_r=k-r}}\prod_{q=1}^rd^{k_q/2}\cC_{k_q/2}\Biggr)\Biggl(d!l\sum_{\substack{l_1,l_2,\ldots,l_r\in2\Z_{\ge0}\\l_1+\cdots+l_r=l-r}}\prod_{q=1}^rd^{l_q/2}\cC_{l_q/2}\Biggr).
\]
Furthermore, $\bar T_n(a)=\{p(1-p)\}^{(k+l)/2}$ for every $\cW_{k,l,(k+l)/2+d-1}^{(2),+}$.
\end{lem}
We defer the proof of the above lemmas to Subsections~\ref{ssec:cW_ks2+-}.
Combining Lemmas~\ref{lem:cW_ks2-} and~\ref{lem:cW_ks2+} with~\eqref{eq:V_kth_moment_6}, we can prove Theorem~\ref{thm:multi_CLT}(1) as follows.
\begin{proof}[Proof of Theorem~\ref{thm:multi_CLT}(1)]
As mentioned in~\eqref{eq:V_kth_moment_5}, when $k+l$ is odd, the conclusion follows.
Hence, suppose that $k+l$ is even.
From~\eqref{eq:V_kth_moment_6}, we have
\begin{align}\label{eq:V_kth_moment_7}
&\lim_{n\to\infty}n^d\{np(1-p)\}\cdot\cov(\langle L_{H_n},x^k\rangle,\langle L_{H_n},x^l\rangle)\nonumber\\
&=\Biggl(\sum_{a=(w^1,w^2)\in\cW_{k,l,(k+l)/2+d-1}^{(2),-}}+\sum_{a=(w^1,w^2)\in\cW_{k,l,(k+l)/2+d-1}^{(2),+}}\Biggr)\lim_{n\to\infty}\frac{\bar T_n(a)}{\{p(1-p)\}^{(k+l)/2-1}}.
\end{align}
From Lemma~\ref{lem:cW_ks2-}, the first term in the right-hand side of~\eqref{eq:V_kth_moment_7} regarding $\cW_{k,l,(k+l)/2+d-1}^{(2),-}$ is equal to
\[
\1_{\{k,l\colon\text{even}\}}(d+1)!\{(k/2)d^{k/2}\cC_{k/2}\}\{(l/2)d^{l/2}\cC_{l/2}\}(2p_\infty-1)^2.
\]
On the other hand, it follows from Lemma~\ref{lem:cW_ks2+} that the second term in the right-hand side of~\eqref{eq:V_kth_moment_7} regarding $\cW_{k,l,(k+l)/2+d-1}^{(2),+}$ is equal to
\[
\sum_{r=3}^{k\wedge l}\frac2r\Biggl(d!k\sum_{\substack{k_1,k_2,\ldots,k_r\in2\Z_{\ge0}\\k_1+\cdots+k_r=k-r}}\prod_{q=1}^rd^{k_q/2}\cC_{k_q/2}\Biggr)\Biggl(d!l\sum_{\substack{l_1,l_2,\ldots,l_r\in2\Z_{\ge0}\\l_1+\cdots+l_r=l-r}}\prod_{q=1}^rd^{l_q/2}\cC_{l_q/2}\Biggr)p_\infty(1-p_\infty),
\]
which completes the proof.
\end{proof}

\subsection{Proof of Lemmas~\ref{lem:cW_ks2-} and~\ref{lem:cW_ks2+}}\label{ssec:cW_ks2+-}
In this subsection, we prove Lemmas~\ref{lem:cW_ks2-} and~\ref{lem:cW_ks2+}.
\begin{proof}[Proof of Lemmas~\ref{lem:cW_ks2-}]
Let $a=(w^1,w^2)\in\cW_{k,l,(k+l)/2+d-1}^{(2),-}$.
Then, $|\Supp_d(a)|=(k+l)/2-1$ and $|\Supp_0(a)|=(k+l)/2+d-1$, particularly, $|\Supp_0(a)|=|\Supp_d(a)|+d$.
From Lemma~\ref{lem:f0-fd}, $X_a$ is a $d$-tree, and necessarily both $X_{w^1}$ and $X_{w^2}$ are also a $d$-tree.
Hence, for $j=1,2$ and $\tau\in\Supp_d(w^j)$, we have $N_{w^j}(\tau)\ge2$.
We then have
\[
k=\sum_{\tau\in\Supp_d(w^1)}N_{w^1}(\tau)\ge2|\Supp_d(w^1)|
\quad\text{and}\quad
l=\sum_{\tau\in\Supp_d(w^2)}N_{w^2}(\tau)\ge2|\Supp_d(w^2)|,
\]
which implies that $|\Supp_d(w^1)|\le\lfloor k/2\rfloor$ and $|\Supp_d(w^2)|\le\lfloor l/2\rfloor$, respectively.
Noting that $\Supp_d(w^1)\cap\Supp_d(w^2)\neq\emptyset$, we obtain
\[
(k+l)/2-1=|\Supp_d(a)|\le|\Supp_d(w^1)|+|\Supp_d(w^2)|-1\le\lfloor k/2\rfloor+\lfloor l/2\rfloor-1.
\]
The above discussion implies that both $k$ and $l$ are even, that $|\Supp_d(w^1)|=k/2$, $|\Supp_d(w^2)|=l/2$, and that $|\Supp_d(w^1)\cap\Supp_d(w^2)|=1$.
Furthermore, since $X_{w^j}$ is $d$-tree, it holds that $|\Supp_0(w^j)|=|\Supp_d(w^j)|+d$ for each $j=1,2$.
Thus, $a=(w^1,w^2)\in\cW_{k,l,(k+l)/2+d-1}^{(2),-}$ is constructed from $w^1\in\cW_{k,k/2+d}$ and $w^2\in\cW_{l,l/2+d}$ such that $X_{w^1}$ and $X_{w^2}$ intersect at exactly one $d$-simplex.
Since there exist $d^{k/2}\cC_{k/2}$ and $d^{l/2}\cC_{l/2}$ ways, respectively, to chose $w^1\in\cW_{k,k/2+d}$ and $w^2\in\cW_{l,l/2+d}$ by Lemma~\ref{lem:KR17_Lem3.11}, $k/2$ and $l/2$ ways of choosing the $d$-simplex in $X_{w^1}$ and $X_{w^2}$ to be glued as the common $d$-simplex, and $(d+1)!$ possible ways for the gluing according to the induced orderings on them, we deduce that
\[
|\cW_{k,l,(k+l)/2+d-1}^{(2),-}|=(d+1)!\{(k/2)d^{k/2}\cC_{k/2}\}\{(l/2)d^{l/2}\cC_{l/2}\}.
\]
Furthermore, noting again Lemma~\ref{lem:KR17_Lem3.11}, we have
\begin{align*}
\bar T_n(a)&=\E\bigl[(\chi_p-p)^4\bigr]\E\bigl[(\chi_p-p)^2\bigr]^{(k+l)/2-2} -\E\bigl[(\chi_p-p)^2\bigr]^{k/2}\E\bigl[(\chi_p-p)^2\bigr]^{l/2}\\
&=p(1-p)\{(1-p)^3+p^3\}\{p(1-p)\}^{(k+l)/2-2} -\{p(1-p)\}^{(k+l)/2}\\
&=\{p(1-p)\}^{(k+l)/2-1}(2p-1)^2
\end{align*}
for any $a\in\cW_{k,l,(k+l)/2+d-1}^{(2),-}$.
\end{proof}

We next turn to consider Lemma~\ref{lem:cW_ks2+}.
Suppose that $a=(w^1,w^2)\in\cW_{k,l,(k+l)/2+d-1}^{(2),+}$.
Then, $|\Supp_0(a)|=(k+l)/2+d-1$ and $|\Supp_d(a)|=(k+l)/2$.
From~\eqref{eq:cW_ks2_1}, we have $N_a(\tau)=2$ for every $\tau\in\Supp_d(a)$.
Since $\Supp_d(w^1)\cap\Supp_d(w^2)\neq\emptyset$, there exists at least one $\tau\in\Supp_d(a)$ such that $N_{w^1}(\tau)=N_{w^2}(\tau)=1$.
Therefore, from~\eqref{eq:barTa},
\begin{equation}\label{eq:bar_Ta_sign}
\bar T_n(a)
=\sgn(a)\prod_{\tau\in\Supp_d(a)}\E\bigl[(\chi_p-p)^2\bigr]
=\pm\{p(1-p)\}^{(k+l)/2}.
\end{equation}
In order to determine the sign in the equation above, we need to understand the detailed structure of $X_a$.
For that purpose, we introduce the notions of bracelets and bracelets with pendant $d$-trees for pure $d$-dimensional simplicial complexes.
\begin{df}[Bracelet]\label{df:bracelet}
Let $r\ge3$.
A pure $d$-dimensional simplicial complex $Z$ is called a \textit{bracelet} of circuit length $r$ if there exist $\rho\in F_{d-2}(Z)$ and distinct vertices $u_1,u_2,\ldots,u_r\in V(Z)\setminus\rho$ such that $Z$ is generated from the $d$-simplices
\[
\rho\cup\{u_1,u_2\},\rho\cup\{u_2,u_3\},\ldots,\rho\cup\{u_{r-1},u_r\},\rho\cup\{u_r,u_1\}.
\]
We sometimes say that $Z$ is a bracelet \textit{formed by} the $(d-2)$-simplex $\rho$ and the distinct vertices $u_1,u_2,\ldots,u_r$.
Furthermore, let $T_1,T_2,\ldots,T_Q$~$(Q\in\Z_{\ge0})$ be $d$-trees such that $Z$ meets each $d$-tree $T_q$ at exactly one $(d-1)$-simplex and distinct $T_q$ and $T_{q'}$ do not intersect except in $Z$.
We call $X\coloneqq Z\cup(T_1\cup T_2\cup\cdots\cup T_Q)$ a \textit{bracelet with pendant $d$-trees}.
We also call $Z$ the \textit{bracelet} of $X$, and each $T_q$ a \textit{pendant $d$-tree} of $X$.
In addition, if $Q=r$, $T_q\cap Z=K(\rho\cup\{u_q\})$ for all $q=1,2,\ldots,r$, and $T_q\cap T_{q'}=K(\rho)$ whenever $q\neq q'\in\{1,2,\ldots,r\}$, we call $X$ a \textit{bracelet with regular pendant $d$-trees}, and each $T_q$ a \textit{regular pendant $d$-tree} of $X$ (see also Figure~\ref{fig:df_bracelet}).
\end{df}
\begin{figure}[H]
\centering
\begin{tikzpicture}[x=8.5mm,y=8.5mm]
\coordinate (u1) at (-.25,1.23);
\coordinate (u2) at (-2,0);
\coordinate (u3) at (-1.75,-1.23);
\coordinate (u4) at (.25,-1.23);
\coordinate (u5) at (2,0);
\coordinate (u6) at (1.75,1.23);
\coordinate (rho) at (0,0);
\coordinate (a) at (0, -2);
\path[fill=red, opacity=.2] (u5)--(u6)--(u1)--(u2)--(u3)--(u4)--cycle;
\draw[line width=.7pt] (u5)--(u6)--(u1)--(u2)--(u3)--(u4)--cycle;
\foreach\P in{u1,u2,u3,u4,u5,u6}\draw[line width=.7pt] (rho)--(\P);
\draw(u1)node[above]{$u_1$};
\draw(u2)node[left]{$u_2$};
\draw(u3)node[below]{$u_3$};
\draw(u4)node[below]{$u_4$};
\draw(u5)node[right]{$u_5$};
\draw(u6)node[above]{$u_6$};
\draw(rho)node[below right]{$\rho$};
\draw(a)node{(a)};
\end{tikzpicture}
\begin{tikzpicture}[x=8.5mm,y=8.5mm]
\coordinate (u1) at (-.25,1.23);
\coordinate (u2) at (-2,0);
\coordinate (u3) at (-1.75,-1.23);
\coordinate (u4) at (.25,-1.23);
\coordinate (u5) at (2,0);
\coordinate (u6) at (1.75,1.23);
\coordinate (rho) at (0,0);
\coordinate (u5u6) at ($2*(u5)!1/2!(u6)+(0,.25)$);
\coordinate (u5u6') at ($2*(u5)!1/2!(u6)-(0,.25)$);
\coordinate (u4u5) at ($2*(u4)!1/2!(u5)$);
\coordinate (u6rho) at ($(rho)!1/2!(u6)+(0,1.73)$);
\coordinate (u6rho') at ($2*(u6rho)!1/2!(u6)$);
\coordinate (u2rho) at ($(rho)!1/2!(u2)+(0,1.73)-(.4,0)$);
\coordinate (u2rho') at ($(rho)!1/2!(u2)+(0,1.73)$);
\coordinate (b) at (0, -2);
\path[fill=red, opacity=.2] (u5)--(u6)--(u1)--(u2)--(u3)--(u4)--cycle;
\draw[line width=.7pt] (u5)--(u6)--(u1)--(u2)--(u3)--(u4)--cycle;
\path[fill=red, opacity=.2] (u5)--(u6)--(u5u6)--cycle; \draw[line width=.7pt] (u5)--(u6)--(u5u6)--cycle;
\path[fill=red, opacity=.2] (u5)--(u6)--(u5u6')--cycle; \draw[line width=.7pt] (u5)--(u6)--(u5u6')--cycle;
\path[fill=red, opacity=.2] (u4)--(u5)--(u4u5)--cycle; \draw[line width=.7pt] (u4)--(u5)--(u4u5)--cycle;
\path[fill=red, opacity=.2] (rho)--(u6)--(u6rho)--cycle; \draw[line width=.7pt] (rho)--(u6)--(u6rho)--cycle;
\path[fill=red, opacity=.2] (u6rho)--(u6)--(u6rho')--cycle; \draw[line width=.7pt] (u6rho)--(u6)--(u6rho')--cycle;
\path[fill=red, opacity=.2] (rho)--(u2)--(u2rho)--cycle; \draw[line width=.7pt] (rho)--(u2)--(u2rho)--cycle;
\path[fill=red, opacity=.2] (rho)--(u2)--(u2rho')--cycle; \draw[line width=.7pt] (rho)--(u2)--(u2rho')--cycle;
\foreach\P in{u1,u2,u3,u4,u5,u6}\draw[line width=.7pt] (rho)--(\P);
\draw(u1)node[above]{$u_1$};
\draw(u2)node[left]{$u_2$};
\draw(u3)node[below]{$u_3$};
\draw(u4)node[below]{$u_4$};
\draw(u5)node[below right]{$u_5$};
\draw(u6)node[above right]{$u_6$};
\draw(rho)node[below right]{$\rho$};
\draw(b)node{(b)};
\end{tikzpicture}
\begin{tikzpicture}[x=8.5mm,y=8.5mm]
\coordinate (u1) at (-.25,1.23);
\coordinate (u2) at (-2,0);
\coordinate (u3) at (-1.75,-1.23);
\coordinate (u4) at (.25,-1.23);
\coordinate (u5) at (2,0);
\coordinate (u6) at (1.75,1.23);
\coordinate (rho) at (0,0);
\coordinate (u5u6) at ($2*(u5)!1/2!(u6)+(0,.25)$);
\coordinate (u5u6') at ($2*(u5)!1/2!(u6)-(0,.25)$);
\coordinate (u4u5) at ($2*(u4)!1/2!(u5)$);
\coordinate (u6rho) at ($(rho)!1/2!(u6)+(0,1.73)$);
\coordinate (u6rho') at ($2*(u6rho)!1/2!(u6)$);
\coordinate (u2rho) at ($(rho)!1/2!(u2)+(0,1.73)-(.4,0)$);
\coordinate (u2rho') at ($(rho)!1/2!(u2)+(0,1.73)$);
\coordinate (c) at (0, -2);
\path[fill=red, opacity=.2] (u5)--(u6)--(u1)--(u2)--(u3)--(u4)--cycle;
\draw[line width=.7pt] (u5)--(u6)--(u1)--(u2)--(u3)--(u4)--cycle;
\path[fill=red, opacity=.2] (rho)--(u6)--(u6rho)--cycle; \draw[line width=.7pt] (rho)--(u6)--(u6rho)--cycle;
\path[fill=red, opacity=.2] (u6rho)--(u6)--(u6rho')--cycle; \draw[line width=.7pt] (u6rho)--(u6)--(u6rho')--cycle;
\path[fill=red, opacity=.2] (rho)--(u2)--(u2rho)--cycle; \draw[line width=.7pt] (rho)--(u2)--(u2rho)--cycle;
\path[fill=red, opacity=.2] (rho)--(u2)--(u2rho')--cycle; \draw[line width=.7pt] (rho)--(u2)--(u2rho')--cycle;
\foreach\P in{u1,u2,u3,u4,u5,u6}\draw[line width=.7pt] (rho)--(\P);
\draw(u1)node[above]{$u_1$};
\draw(u2)node[left]{$u_2$};
\draw(u3)node[below]{$u_3$};
\draw(u4)node[below]{$u_4$};
\draw(u5)node[below right]{$u_5$};
\draw(u6)node[above right]{$u_6$};
\draw(rho)node[below right]{$\rho$};
\draw(u2rho)node[above]{$A$};
\draw(u2rho')node[above]{$B$};
\draw(c)node{(c)};
\end{tikzpicture}
\caption{(a) An illustration of a bracelet of circuit length $6$ in the case of $d=2$.
(b) An illustration of a bracelet (of circuit length $6$) with pendant $d$-trees in the case of $d=2$.
Here, four (nontrivial) pendant $d$-trees are attached to the bracelet part.
(c) An illustration of a bracelet (of circuit length $6$) with regular pendant $d$-trees in the case of $d=2$.
For example, the regular pendant $d$-trees $T_1$ and $T_2$ are the trivial $d$-tree $K(\{u_1,\rho\})$ and the $d$-tree generated from two $d$-simplices $\{u_2,\rho,A\}$ and $\{u_2,\rho,B\}$, respectively.}
\label{fig:df_bracelet}
\end{figure}
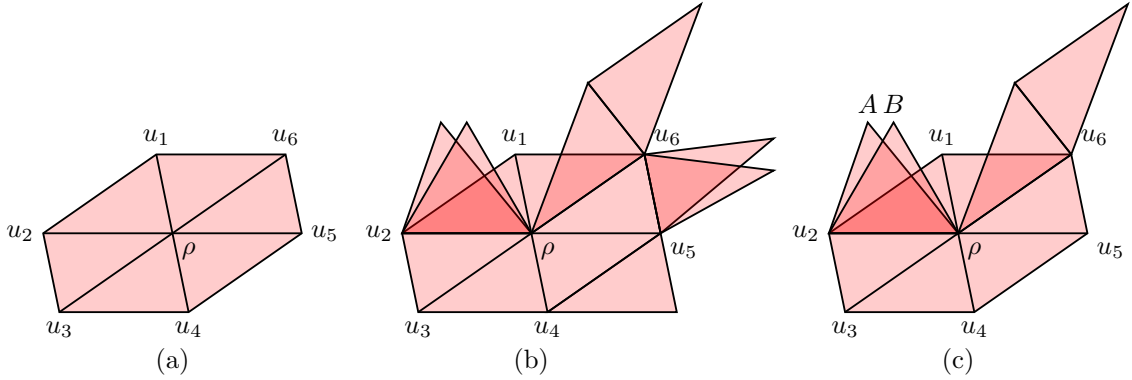
The following lemma is a useful property of bracelets with pendant $d$-trees (see also Figure~\ref{fig:d-tree}).
\begin{lem}\label{lem:d-tree}
Let $X$ be a $d$-tree, and let $w_0,w_1,\ldots,w_d\in V(X)$ be distinct vertices.
Suppose that $\tau\coloneqq\{w_0,w_1,\ldots,w_d\}\notin X$.
Then, the following hold.
\begin{enumerate}
\item There exist $(d-1)$-dimensional faces $\sg$ and $\sg'$ of $\tau$ such that $X\cap K(\tau)\subset K(\sg)\cup K(\sg')$.
\item If there exist $(d-1)$-dimensional faces $\sg\neq\sg'$ of $\tau$ such that $X\cap K(\tau)=K(\sg)\cup K(\sg')$, then $X\cup K(\tau)$ is a bracelet with pendant $d$-trees.
\end{enumerate}
\end{lem}
\begin{proof}
Let $v_t$'s and $\sg_{t-1}$'s~($t=1,2,\ldots,f_d(X)$) be the vertices and $(d-1)$-simplices in a generating process $X(0)\subsetneq X(1)\subsetneq\cdots\subsetneq X(f_d(X))=X$ of $X$, as described in~\eqref{eq:process}.

(1)~Note first that the number of vertices in $X(t)\cap K(\tau)$, i.e., $|V(X(t)\cap K(\tau))|$, either increases by one or does not change at each step.
Since $|V(X(0)\cap K(\tau))|\le d$ and $|V(X\cap K(\tau))|=d+1$, there exist the first steps $t_1$ and $t_2$ ($0\le t_1<t_2\le f_d(X)$) such that $|V(X(t_1)\cap K(\tau))|=d$ and $|V(X(t_2)\cap K(\tau))|=d+1$, respectively.
We define a $(d-1)$-dimensional face $\sg$ of $\tau$ as $V(X(t_1)\cap K(\tau))$, equivalently $\sg\coloneqq\tau\setminus\{v_{t_2}\}$.
Then, $X(t_1)\cap K(\tau)
\subset K(\sg)$.
Since the step $t_2$ is the first step at which the vertex $v_{t_2}$ is added in the process, it still holds that
\[
X(t_2-1)\cap K(\tau)
\subset K(\sg).
\]
Since $\tau\notin X$, the intersection of $\tau$ and the new $d$-simplex $\sg_{t_2-1}\cup\{v_{t_2}\}$ added at the step $t_2$ is contained in a $(d-1)$-dimensional face $\sg'$ of $\tau$.
Therefore,
\[
X(t_2)\cap K(\tau)
=\{X(t_2-1)\cup K(\sg_{t_2-1}\cup\{v_{t_2}\})\}\cap K(\tau)
\subset(X(t_2-1)\cap K(\tau))\cup K(\sg').
\]
Furthermore, since $X$ is a $d$-tree, $v_t\notin X(t-1)$ for all $t=1,2,\ldots,f_d(X)$.
Hence, noting that $\tau\subset V(X(t_2))$, we have $v_t\notin\tau$ for all $t=t_2+1,t_2+2,\ldots,f_d(X)$, which implies that
\[
X(t_2)\cap K(\tau)=X(t_2+1)\cap K(\tau)=\cdots=X(f_d(X))\cap K(\tau)=X\cap K(\tau).
\]
Combining the above equations, we obtain $X\cap K(\tau)\subset K(\sg)\cup K(\sg')$.

(2)~Since $X$ is a $d$-tree, it suffices to find a bracelet contained in $X\cup K(\tau)$.
For every simplex $\hat\sg\in X$, we denote by $t(\hat\sg)$ the first step such that $\hat\sg$ is added in the process:
\[
t(\hat\sg)\coloneqq\min\{t=0,1,\ldots,f_d(X)\mid\hat\sg\in X(t)\}.
\]
Set $\rho=\sg\cap\sg'$, and let $u_1\in V(X)\setminus\rho$ be a vertex such that $\rho\cup\{u_1\}=\sg$.
If $t(\rho)<t(\sg)$, then $\rho\subset\sg_{t(\sg)-1}\neq\sg$, which implies that the new vertex $v_{t(\sg)}$ added at the step $t(\sg)$ must be identical to $u_1$.
Let $u_2\in V(X)\setminus\rho$ be a vertex such that $\sg_{t(\sg)-1}=\rho\cup\{u_2\}$.
Note that $\rho\cup\{u_1,u_2\}=\sg_{t(\sg)-1}\cup\{v_{t(\sg)}\}$ is a $d$-simplex in $X$ and that $t(\rho\cup\{u_2\})<t(\rho\cup\{u_1\})$.
For the same reasoning, if $t(\rho)<t(\rho\cup\{u_2\})$, then there exists a vertex $u_3\in V(X)\setminus\rho$ such that $\rho\cup\{u_2,u_3\}$ is a $d$-simplex in $X$ and that $t(\rho\cup\{u_3\})<t(\rho\cup\{u_2\})$.
By iterating this procedure whenever $t(\rho)<t(\rho\cup\{u_i\})$, we obtain distinct vertices $u_1,u_2,\ldots,u_r\in V(X)\setminus\rho$ such that $\rho\cup\{u_1\}=\sg$, $t(\rho\cup\{u_r\})=t(\rho)$, and that each of $\rho\cup\{u_1,u_2\},\rho\cup\{u_2,u_3\},\ldots,\rho\cup\{u_{r-1},u_r\}$ is a $d$-simplex in $X$.
In the same way, we also obtain distinct vertices $u'_1,u'_2,\ldots,u'_{r'}\in V(X)\setminus\rho$ such that $\rho\cup\{u'_1\}=\sg'$, $t(\rho\cup\{u'_{r'}\})=t(\rho)$, and that each of $\rho\cup\{u'_1,u'_2\},\rho\cup\{u'_2,u'_3\},\ldots,\rho\cup\{u'_{r'-1},u_{r'}\}$ is a $d$-simplex in $X$.
Suppose that $\{u_1,u_2,\ldots,u_r\}\cap\{u'_1,u'_2,\ldots,u'_{r'}\}\neq\emptyset$.
Let $r''\coloneqq\min\{i\in\{1,2,\ldots,r'\}\mid u'_i\in\{u_1,u_2,\ldots,u_r\}\}$.
Then, the $(d-2)$-simplex $\rho$ and the distinct vertices $u_1,u_2,\ldots,u_r,u'_{r''-1},u'_{r''-2},\ldots,u'_1$ form a bracelet.
Next, suppose that $\{u_1,u_2,\ldots,u_r\}\cap\{u'_1,u'_2,\ldots,u'_{r'}\}=\emptyset$.
Then, $\rho\cup\{u_r,u'_{r'}\}$ is a $d$-simplex in $X$ because $t(\rho\cup\{u_r\})=t(\rho)=t(\rho\cup\{u'_{r'}\})$.
Therefore, the $(d-2)$-simplex $\rho$ and the distinct vertices $u_1,u_2,\ldots,u_r,u'_{r'},u'_{r'-1},\ldots,u'_1$ form a bracelet.
\end{proof}
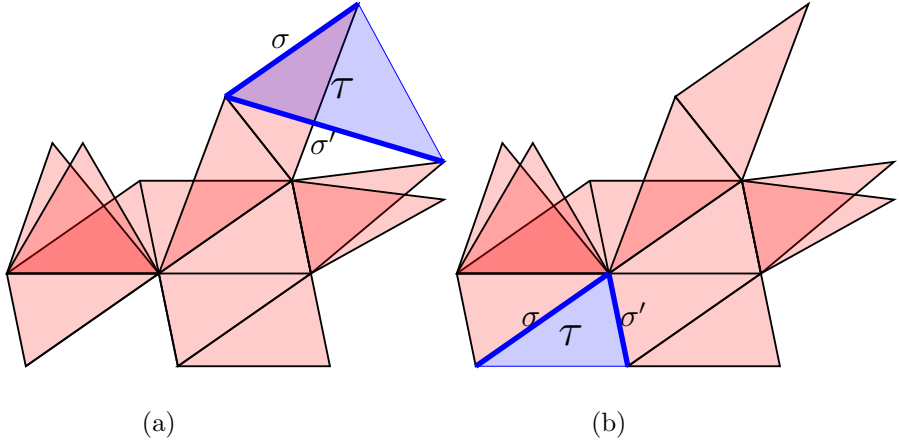
\begin{figure}[H]
\centering
\begin{tikzpicture}
\coordinate (u1) at (-.25,1.23);
\coordinate (u2) at (-2,0);
\coordinate (u3) at (-1.75,-1.23);
\coordinate (u4) at (.25,-1.23);
\coordinate (u5) at (2,0);
\coordinate (u6) at (1.75,1.23);
\coordinate (rho) at (0,0);
\coordinate (u5u6) at ($2*(u5)!1/2!(u6)+(0,.25)$);
\coordinate (u5u6') at ($2*(u5)!1/2!(u6)-(0,.25)$);
\coordinate (u4u5) at ($2*(u4)!1/2!(u5)$);
\coordinate (u6rho) at ($(rho)!1/2!(u6)+(0,1.73)$);
\coordinate (u6rho') at ($2*(u6rho)!1/2!(u6)$);
\coordinate (u2rho) at ($(rho)!1/2!(u2)+(0,1.73)-(.4,0)$);
\coordinate (u2rho') at ($(rho)!1/2!(u2)+(0,1.73)$);
\coordinate (a) at (0, -2);
\path[fill=red, opacity=.2] (u5)--(u6)--(u1)--(u2)--(u3)--(rho)--(u4)--cycle;
\draw[line width=.7pt] (u5)--(u6)--(u1)--(u2)--(u3)--(rho)--(u4)--cycle;
\path[fill=red, opacity=.2] (u5)--(u6)--(u5u6)--cycle; \draw[line width=.7pt] (u5)--(u6)--(u5u6)--cycle;
\path[fill=red, opacity=.2] (u5)--(u6)--(u5u6')--cycle; \draw[line width=.7pt] (u5)--(u6)--(u5u6')--cycle;
\path[fill=red, opacity=.2] (u4)--(u5)--(u4u5)--cycle; \draw[line width=.7pt] (u4)--(u5)--(u4u5)--cycle;
\path[fill=red, opacity=.2] (rho)--(u6)--(u6rho)--cycle; \draw[line width=.7pt] (rho)--(u6)--(u6rho)--cycle;
\path[fill=red, opacity=.2] (u6rho)--(u6)--(u6rho')--cycle; \draw[line width=.7pt] (u6rho)--(u6)--(u6rho')--cycle;
\path[fill=red, opacity=.2] (rho)--(u2)--(u2rho)--cycle; \draw[line width=.7pt] (rho)--(u2)--(u2rho)--cycle;
\path[fill=red, opacity=.2] (rho)--(u2)--(u2rho')--cycle; \draw[line width=.7pt] (rho)--(u2)--(u2rho')--cycle;
\foreach\P in{u1,u2,u3,u4,u5,u6}\draw[line width=.7pt] (rho)--(\P);
\path[fill=blue, opacity=.2] (u6rho)--(u6rho')--(u5u6)--cycle; \draw[line width=.35pt, blue] (u6rho)--(u6rho')--(u5u6)--cycle;
\draw[line width=2pt, blue] (u6rho)--(u5u6);
\draw[line width=2pt, blue] (u6rho')--(u6rho);
\coordinate (tau) at ($.333*($(u6rho)+(u6rho')+(u5u6)$)$);
\coordinate (sg) at ($(u6rho)!1/2!(u6rho')+(.1,-.1)$);
\coordinate (sg') at ($(u6rho)!1/2!(u5u6)+(-.15,.1)$);
\draw(tau)node{\LARGE$\tau$};
\draw(sg)node[above left]{\large$\sg$};
\draw(sg')node[below]{\large$\sg'$};
\draw(a)node{(a)};
\end{tikzpicture}
\begin{tikzpicture}
\coordinate (u1) at (-.25,1.23);
\coordinate (u2) at (-2,0);
\coordinate (u3) at (-1.75,-1.23);
\coordinate (u4) at (.25,-1.23);
\coordinate (u5) at (2,0);
\coordinate (u6) at (1.75,1.23);
\coordinate (rho) at (0,0);
\coordinate (u5u6) at ($2*(u5)!1/2!(u6)+(0,.25)$);
\coordinate (u5u6') at ($2*(u5)!1/2!(u6)-(0,.25)$);
\coordinate (u4u5) at ($2*(u4)!1/2!(u5)$);
\coordinate (u6rho) at ($(rho)!1/2!(u6)+(0,1.73)$);
\coordinate (u6rho') at ($2*(u6rho)!1/2!(u6)$);
\coordinate (u2rho) at ($(rho)!1/2!(u2)+(0,1.73)-(.4,0)$);
\coordinate (u2rho') at ($(rho)!1/2!(u2)+(0,1.73)$);
\coordinate (b) at (0, -2);
\path[fill=red, opacity=.2] (u5)--(u6)--(u1)--(u2)--(u3)--(rho)--(u4)--cycle;
\draw[line width=.7pt] (u5)--(u6)--(u1)--(u2)--(u3)--(rho)--(u4)--cycle;
\path[fill=red, opacity=.2] (u5)--(u6)--(u5u6)--cycle; \draw[line width=.7pt] (u5)--(u6)--(u5u6)--cycle;
\path[fill=red, opacity=.2] (u5)--(u6)--(u5u6')--cycle; \draw[line width=.7pt] (u5)--(u6)--(u5u6')--cycle;
\path[fill=red, opacity=.2] (u4)--(u5)--(u4u5)--cycle; \draw[line width=.7pt] (u4)--(u5)--(u4u5)--cycle;
\path[fill=red, opacity=.2] (rho)--(u6)--(u6rho)--cycle; \draw[line width=.7pt] (rho)--(u6)--(u6rho)--cycle;
\path[fill=red, opacity=.2] (u6rho)--(u6)--(u6rho')--cycle; \draw[line width=.7pt] (u6rho)--(u6)--(u6rho')--cycle;
\path[fill=red, opacity=.2] (rho)--(u2)--(u2rho)--cycle; \draw[line width=.7pt] (rho)--(u2)--(u2rho)--cycle;
\path[fill=red, opacity=.2] (rho)--(u2)--(u2rho')--cycle; \draw[line width=.7pt] (rho)--(u2)--(u2rho')--cycle;
\foreach\P in{u1,u2,u3,u4,u5,u6}\draw[line width=.7pt] (rho)--(\P);
\path[fill=blue, opacity=.2] (u3)--(u4)--(rho)--cycle; \draw[line width=.35pt, blue] (u3)--(u4)--(rho)--cycle;
\draw[line width=2pt, blue] (u3)--(rho);
\draw[line width=2pt, blue] (u4)--(rho);
\coordinate (tau) at ($.333*($(rho)+(u3)+(u4)$)$);
\coordinate (sg) at ($(rho)!1/2!(u3)+(.1,-.2)$);
\coordinate (sg') at ($(rho)!1/2!(u4)+(-.1,-.2)$);
\draw(tau)node{\LARGE$\tau$};
\draw(sg)node[above left]{\large$\sg$};
\draw(sg')node[above right]{\large$\sg'$};
\draw(b)node{(b)};
\end{tikzpicture}
\caption{Illustrations for Lemma~\ref{lem:d-tree} in the case of $d=2$.
In each of (a) and (b), the red part indicates the $d$-tree $X$.
(a) The $(d-1)$-dimensional faces $\sg$ and $\sg'$ of $\tau$ satisfy that $X\cap K(\tau)\subset K(\sg)\cup K(\sg')$.
(b) The $(d-1)$-dimensional faces $\sg\neq\sg'$ of $\tau$ satisfy that $X\cap K(\tau)=K(\sg)\cup K(\sg')$.
In this case, $X\cup K(\tau)$ is actually a bracelet with pendant $d$-trees.}
\label{fig:d-tree}
\end{figure}

Coming back to the pure $d$-dimensional simplicial complex $X_a$ for $a\in\cW_{k,l,(k+l)/2+d-1}^{(2),+}$, we now determine the detailed structure of $X_a$.
\begin{lem}\label{lem:bracelet}
Let $k,l\ge2$ such that $k+l$ is even, and let $a=(w^1,w^2)\in\cW_{k,l,(k+l)/2+d-1}^{(2),+}$.
Then, $X_a$ is a bracelet with regular pendant $d$-trees.
\end{lem}
\begin{proof}
Noting that $X_a$ is a strongly connected pure $d$-dimensional simplicial complex, we fix a generating process $X(0)\subsetneq X(1)\subsetneq\cdots\subsetneq X(k)=X_a$ of $X_a$ as described in~\eqref{eq:process}.
Since $|\Supp_0(a)|=(k+l)/2+d-1=(|\Supp_d(a)|+d)-1$, there exists a step $t'\in\{1,2,\ldots,(k+l)/2\}$ in the generating process such that $v_{t'}\in X(t'-1)$ and $v_t\notin V(X(t-1))$ for all $t\in\{1,2,\ldots,(k+l)/2\}\setminus\{t'\}$.
Since $X(t'-1)$ is a $d$-tree and $\sg_{t'-1}\cup\{v_{t'}\}\notin X(t'-1)$, it follows from Lemma~\ref{lem:d-tree}(1) that there exist $(d-1)$-dimensional faces $\sg$ and $\sg'$ of $\tau$ such that $X(t'-1)\cap K(\sg_{t'-1}\cup\{v_{t'}\})\subset K(\sg)\cup K(\sg')$.
In particular, $f_{d-1}(X(t'-1)\cap K(\sg_{t'-1}\cup\{v_{t'}\}))\in\{1,2\}$.
Assume that $f_{d-1}(X(t'-1)\cap K(\sg_{t'-1}\cup\{v_{t'}\}))=1$.
In this case, the structure of $X_a$ is essentially $d$-tree in the sense that the inclusionwise relationship between $(d-1)$- and $d$-simplices in $X_a$ is identical to that of a $d$-tree.
Hence, $N_{w^j}(\tau)\ge2$ holds for $j=1,2$ and $\tau\in\Supp_d(w^j)$.
Therefore,
\[
k=\sum_{\tau\in\Supp_d(w^1)}N_{w^1}(\tau)\ge2|\Supp_d(w^1)|
\quad\text{and}\quad
l=\sum_{\tau\in\Supp_d(w^2)}N_{w^2}(\tau)\ge2|\Supp_d(w^2)|,
\]
which implies that $|\Supp_d(w^1)|\le\lfloor k/2\rfloor$ and $|\Supp_d(w^2)|\le\lfloor l/2\rfloor$, respectively.
Thus, we have
\[
(k+l)/2=|\Supp_d(a)| \le|\Supp_d(w^1)|+|\Supp_d(w^2)|-1 \le\lfloor k/2\rfloor+\lfloor l/2\rfloor-1\le(k+l)/2-1,
\]
which contradicts.
Consequently, it must hold that $X(t'-1)\cap K(\sg_{t'-1}\cup\{v_{t'}\})=K(\sg)\cup K(\sg')$, which implies that $X_{t'}$ is a bracelet with pendant $d$-trees from Lemma~\ref{lem:d-tree}(2).
Clearly, $X_a$ is also a bracelet with pendant $d$-trees by taking account of the generating process of $X_a$ after step $t'$.

Finally, we show the regularity of the pendant $d$-trees of $X_a$.
Let $Z_a$ be the bracelet of $X_a$, and suppose that $Z_a$ is generated from $\rho\in F_{d-2}(Z)$ and distinct vertices $u_1,u_2,\ldots,u_r\in V(Z)\setminus\rho$ as described in Definition~\ref{df:bracelet}.
Assume that there exist $d$-simplices $\tau_b\in Z_a$ and $\tau_p\in X_a\setminus Z_a$ such that $\tau_b\cap\tau_p$ is a $(d-1)$-simplex that coincides none of $\rho\cup\{u_1\},\rho\cup\{u_2\},\ldots,\rho\cup\{u_r\}$.
We may assume that $N_{w^1}(\tau_p)\ge1$ without loss of generality.
Then, it holds that $N_{w^1}(\tau_b)\ge1$.
Indeed, if $N_{w^1}(\tau_b)=0$, then $X_{w^1}$ is a $d$-tree.
Therefore, $N_{w^1}(\tau)\ge2$ for any $\tau\in\Supp_d(w^1)$.
Since $N_a(\tau)=2$ for any $\tau\in\Supp_d(a)$, we have $\Supp_d(w^1)\cap\Supp_d(w^2)=\emptyset$, which is not the case.
In the same reasoning, we have $N_{w^2}(\tau_b)\ge1$.
Consequently, we obtain $N_{w^1}(\tau_b)=N_{w^2}(\tau_b)=1$ (see also Figure~\ref{fig:bracelet}).
However, this is not the case because the closed $(n,d)$-word $w^1$ must visit $\tau_b$ at least twice, i.e., $N_{w^1}(\tau_b)\ge2$.
\end{proof}
\begin{figure}[H]
\centering
\begin{tikzpicture}[x=10mm,y=10mm]
\coordinate (u1) at (-.25,1.23);
\coordinate (u2) at (-2,0);
\coordinate (u3) at (-1.75,-1.23);
\coordinate (u4) at (.25,-1.23);
\coordinate (u5) at (2,0);
\coordinate (u6) at (1.75,1.23);
\coordinate (rho) at (0,0);
\coordinate (u5u6) at ($2*(u5)!1/2!(u6)$);
\coordinate (u5u6') at ($(u5)+(2.3,.2)$);
\coordinate (u5u6'') at ($(u5)+(2,-.2)$);
\coordinate (u4u5) at ($2*(u4)!1/2!(u5)$);
\coordinate (u6rho) at ($(rho)!1/2!(u6)+(0,1.73)$);
\coordinate (u6rho') at ($2*(u6rho)!1/2!(u6)$);
\coordinate (u2rho) at ($(rho)!1/2!(u2)+(0,1.73)-(.4,0)$);
\coordinate (u2rho') at ($(rho)!1/2!(u2)+(0,1.73)$);
\coordinate (taub) at ($.333*($(u5)+(u6)$)$);
\coordinate (taup) at ($.333*($(u5)+(u6)+(u5u6)$)$);
\coordinate (Xa) at ($(u1)+(0,1)$);
\coordinate (A') at ($.2*($3*(u5)+2*(u5u6')$)$);
\coordinate (A) at ($.2*($3*(u5)+2*(u5u6)$)$);
\coordinate (B) at ($.2*($3*(u5)+2*(u6)$)$);
\coordinate (C) at ($.2*($3*(u5)+2*(rho)$)$);
\coordinate (C') at ($.2*($3*(u4)+2*(rho)$)$);
\coordinate (D') at ($.2*($3*(u6rho)+2*(u6)$)$);
\coordinate (D) at ($.2*($2*(u6)+3*(rho)$)$);
\coordinate (E) at ($.2*($2*(u5)+3*(rho)$)$);
\coordinate (E') at ($.2*($2*(u4)+3*(rho)$)$);
\path[fill=red, opacity=.2] (u5)--(u6)--(u1)--(u2)--(u3)--(u4)--cycle;
\draw[line width=.7pt] (u5)--(u6)--(u1)--(u2)--(u3)--(u4)--cycle;
\path[fill=red, opacity=.2] (u5)--(u6)--(u5u6)--cycle; \draw[line width=.7pt] (u5)--(u6)--(u5u6)--cycle;
\path[fill=red, opacity=.2] (u5)--(u5u6)--(u5u6')--cycle; \draw[line width=.7pt] (u5)--(u5u6)--(u5u6')--cycle;
\path[fill=red, opacity=.2] (u5)--(u5u6)--(u5u6'')--cycle; \draw[line width=.7pt] (u5)--(u5u6)--(u5u6'')--cycle;
\path[fill=red, opacity=.2] (rho)--(u6)--(u6rho)--cycle; \draw[line width=.7pt] (rho)--(u6)--(u6rho)--cycle;
\path[fill=red, opacity=.2] (u6rho)--(u6)--(u6rho')--cycle; \draw[line width=.7pt] (u6rho)--(u6)--(u6rho')--cycle;
\path[fill=red, opacity=.2] (rho)--(u2)--(u2rho)--cycle; \draw[line width=.7pt] (rho)--(u2)--(u2rho)--cycle;
\path[fill=red, opacity=.2] (rho)--(u2)--(u2rho')--cycle; \draw[line width=.7pt] (rho)--(u2)--(u2rho')--cycle;
\foreach\P in{u1,u2,u3,u4,u5,u6}\draw[line width=.7pt] (rho)--(\P);
\draw[line width=1pt, red] (A)--(B)--(C);
\draw[line width=1pt, red, dashed] (A)--(A');
\draw[line width=1pt, red, dashed, ->] (C)--(C');
\draw[line width=1pt, blue] (D)--(E);
\draw[line width=1pt, blue, dashed] (D)--(D');
\draw[line width=1pt, blue, dashed, ->] (E)--(E');
\foreach\P in{A,B,C}\fill[red](\P)circle(2pt);
\foreach\P in{D,E}\fill[blue](\P)circle(2pt);
\draw(u1)node[above]{$u_1$};
\draw(u2)node[left]{$u_2$};
\draw(u3)node[below]{$u_3$};
\draw(u4)node[below]{$u_4$};
\draw(u5)node[below right]{$u_5$};
\draw(u6)node[above right]{$u_6$};
\draw(rho)node[below right]{$\rho$};
\draw(taub)node{\Large$\tau_b$};
\draw(taup)node{\Large$\tau_p$};
\draw(Xa)node{\Large$X_a$};
\end{tikzpicture}
\caption{An illustration of the $(n,d)$-words $w^1$ (red line) and $w^2$ (blue line) when we assume the existence of the $\tau_b$ and $\tau_p$ in the proof of Lemma~\ref{lem:bracelet} in the case of $d=2$.}
\label{fig:bracelet}
\end{figure}
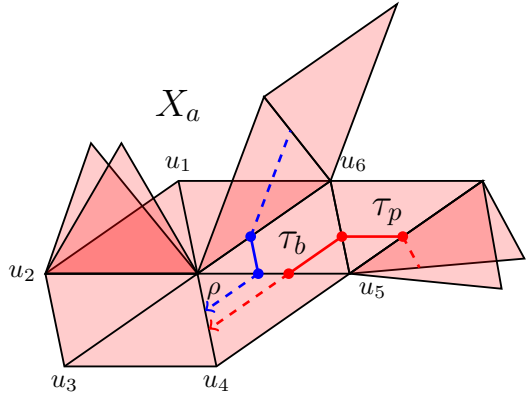

Now, we are ready to prove Lemma~\ref{lem:cW_ks2+}.
\begin{proof}[Proof of Lemma~\ref{lem:cW_ks2+}]
Suppose that $a=(w^1,w^2)\in\cW_{k,l,(k+l)/2+d-1}^{(2),+}$.
Recall that $N_a(\tau)=2$ for every $\tau\in\Supp_d(a)$.
From Lemma~\ref{lem:bracelet}, $X_a$ is a bracelet with regular pendant $d$-trees.
Let $Z_a$ denote the bracelet of $X_a$.
For each $\tau\in\Supp_d(a)$, we first claim that $\tau\in Z_a$ if and only if $N_{w^1}(\tau)=N_{w^2}(\tau)=1$.
For the one direction, suppose that $\tau\in Z_a$.
Then, since $X_a$ is a bracelet with regular pendant $d$-trees, $X_a\setminus\{\tau\}$ consists of one $d$-tree and $d-1$ maximal $(d-1)$-simplices.
Assume that the $(n,d)$-word $w^j$ fails to visit $\tau$, i.e., $N_{w^j}(\tau)=0$.
Then, $X_{w^j}$ must be a $d$-tree, which implies that $N_{w^j}(\tau')\ge2$ for all $\tau'\in\Supp_d(w^j)$.
Thus, $\Supp_d(w^1)\cap\Supp_d(w^2)=\emptyset$, contradicting the definition of $\cW_{k,l,(k+l)/2+d-1}^{(2)}$.
For the other direction, suppose that $\tau$ belongs to a regular pendant $d$-tree of $X_a$.
Then, one of the closed $(n,d)$-words $w^1$ and $w^2$ visits $\tau$ at least twice.
Consequently, we obtain the above claim.

Let $r\ge3$ be fixed, and let $\cW_{k,l,(k+l)/2+d-1}^{(2),+,r}$ denote the set of all $a=(w^1,w^2)\in\cW_{k,l,(k+l)/2+d-1}^{(2),+}$ whose bracelet $Z_a$ is of circuit length $r$.
We may assume without loss of generality that the bracelet $Z_a$ of each $a=(w^1,w^2)\in\cW_{k,l,(k+l)/2+d-1}^{(2),+,r}$ is supported on $[d-1+r]$ and is formed by a $(d-2)$-simplex $[d-1]\in F_{d-2}(Z_a)$ and the distinct vertices $d,d+1,\ldots,d-1+r$.
Then, the $a=(w^1,w^2)\in\cW_{k,l,(k+l)/2+d-1}^{(2),+,r}$ possesses the following data:
the closed $(n,d)$-words $\{w^j_q\}_{q=1}^r$ restricted to the regular pendant $d$-trees attached to the $(d-1)$-simplices $\{[d-1]\cup\{d-1+q\}\}_{q=1}^r$ associated with each closed $(n,d)$-word $w^j$, what number the $(d-1)$-simplex $[d]$ appears in each $(n,d)$-word $w^j$, the induced ordering on the $(d-1)$-simplex $[d]$ by each $(n,d)$-word $w^j$ (recall Definition~\ref{df:induced_ordering}), and whether $Z_a$ is traversed by the closed $(n,d)$-words $w^1$ and $w^2$ in the same or in opposing directions.
Here, the ordering of the initial $(d-1)$-simplex $[d-1]\cup\{d-1+q\}$ of each $w^j_q$ is given by the induced ordering on $[d-1]\cup\{d-1+q\}$ by the $(n,d)$-word $w^j$.
Meanwhile, we can verify that these data are enough to recover the original $a\in\cW_{k,l,(k+l)/2+d-1}^{(2),+,r}$.
Since $r$ number of such data generates the same $a\in\cW_{k,l,(k+l)/2+d-1}^{(2),+,r}$ due to the rotation of the bracelet, there exist exactly
\[
\frac2r\Biggl(d!k\sum_{\substack{k_1,k_2,\ldots,k_r\in2\Z_{\ge0}\\k_1+\cdots+k_r=k-r}}\prod_{q=1}^rd^{k_q/2}\cC_{k_q/2}\Biggr)\Biggl(d!l\sum_{\substack{l_1,l_2,\ldots,l_r\in2\Z_{\ge0}\\l_1+\cdots+l_r=l-r}}\prod_{q=1}^rd^{l_q/2}\cC_{l_q/2}\Biggr)
\]
elements in $\cW_{k,l,(k+l)/2+d-1}^{(2),+,r}$ from Lemma~\ref{lem:KR17_Lem3.11}(3).
Furthermore, using also Lemma~\ref{lem:KR17_Lem3.11}(2), we can determine the sign in the right-hand side of~\eqref{eq:bar_Ta_sign}:
\[
\bar T_n(a)=\sgn(a)\{p(1-p)\}^{(k+l)/2}=\{p(1-p)\}^{(k+l)/2}.\qedhere
\]
\end{proof}

\section{Multivariate central limit theorem for moments}\label{sec:multivariate_CLT}
In this section, we prove Theorem~\ref{thm:multi_CLT}(2).
For any fixed $K\in\N$, the matrix $\Sg_K=\{\sg(k,l)\}_{0\le k,l\le K}$ is symmetric and positive semidefinite since
\[
\sg(k,l)
=\lim_{n\to\infty}n^d\{np(1-p)\}\cdot\cov(\langle L_{H_n},x^k\rangle,\langle L_{H_n},x^l\rangle)
\]
for any $k,l\in\Z_{\ge0}$ from Theorem~\ref{thm:multi_CLT}(1).
Hence, there exists a mean-zero Gaussian process $\{W_k\}_{k=0}^\infty$ such that $\E[W_kW_l]=\sg(k,l)$ for any $k,l\in\Z_{\ge0}$.
Our primary goal in this section is to prove the following lemma, which reduces to Theorem~\ref{thm:multi_CLT}(1) when $h=2$.
\begin{lem}\label{lem:multi_moment}
Let $h\ge2$ and $k_1,k_2,\ldots,k_h\in\Z_{\ge0}$.
If $\lim_{n\to\infty}np(1-p)=\infty$, then
\begin{equation}\label{eq:multi_moment}
\lim_{n\to\infty}[n^d\{np(1-p)\}]^{h/2}\cdot\E\Biggl[\prod_{j=1}^h(\langle L_{H_n},x^{k_j}\rangle-\E\langle L_{H_n},x^{k_j}\rangle)\Biggr]
=\E[W_{k_1}W_{k_2}\cdots W_{k_h}].
\end{equation}
\end{lem}
Once the above lemma is proved, the multivariate version of Carleman's theorem implies that for any $K\in\Z_{\ge0}$,
\begin{equation}\label{eq:multi_moment_0}
\Bigl\{\sqrt{n^d\{np(1-p)\}}\cdot(\langle L_{H_n},x^k\rangle-\E\langle L_{H_n},x^k\rangle)\Bigr\}_{k=0}^K
\xrightarrow[n\to\infty]{d}(W_0,W_1,\ldots,W_K).
\end{equation}
Since $(W_0,W_1,\ldots,W_K)\sim\cN(0,\Sg_K)$, Theorem~\ref{thm:multi_CLT}(2) follows.
It follows from~\eqref{eq:multi_moment_0} that for any real-valued polynomial function $f(x)=\sum_{k=0}^Ka_kx^k$,
\begin{equation}\label{eq:multi_moment_00}
\sqrt{n^d\{np(1-p)\}}\cdot(\langle L_{H_n},f\rangle-\E\langle L_{H_n},f\rangle)
\xrightarrow[n\to\infty]{d}\sum_{k=0}^Ka_kW_k\sim\cN\biggl(0,\sum_{k,l=0}^Ka_k\sg(k,l)a_l\biggr).
\end{equation}
Furthermore, by Theorem~\ref{thm:multi_CLT}(1), we have
\begin{align}\label{eq:multi_moment_000}
n^d\{np(1-p)\}\cdot\Var(\langle L_{Y_n},f\rangle)
&=n^d\{np(1-p)\}\cdot\Var\Biggl(\sum_{k=0}^Ka_k\langle L_{Y_n},x^k\rangle\Biggr)\nonumber\\
&=\sum_{k=0}^K\sum_{l=0}^Ka_ka_ln^d\{np(1-p)\}\cdot\cov(\langle L_{Y_n},x^k\rangle,\langle L_{Y_n},x^l\rangle)\nonumber\\
&\xrightarrow[n\to\infty]{}\sum_{k,l=0}^Ka_k\sg(k,l)a_l,
\end{align}
which completes the proof of Corollary~\ref{cor:poly-type_CLT}.

Before starting the proof of Lemma~\ref{lem:multi_moment}, we remark an expression of the right-hand side of~\eqref{eq:multi_moment}, which is useful in the proof below.
\begin{rem}\label{rem:multi_moment}
When $h$ is odd, $\E[W_{k_1}W_{k_2}\cdots W_{k_h}]=0$.
Indeed, since $(W_{k_1},W_{k_2},\ldots,W_{k_h})$ has the same distribution as that of $(-W_{k_1},-W_{k_2},\ldots,-W_{k_h})$, we have
\[
\E[W_{k_1}W_{k_2}\cdots W_{k_h}]
=\E[(-W_{k_1})(-W_{k_2})\cdots(-W_{k_h})]
=-\E[W_{k_1}W_{k_2}\cdots W_{k_h}].
\]
When $h$ is even, Isserlis' theorem yields
\begin{equation}\label{eq:multi_moment_1}
\E[W_{k_1}W_{k_2}\cdots W_{k_h}]
=\sum_{\substack{\{j(c),j'(c)\}_{c=1}^{h/2}\colon\\\text{matching on $[h]$}}}\prod_{c=1}^{h/2}\sg(k_{j(c)},k_{j'(c)}).
\end{equation}
Note that if $k_1+k_2+\cdots+k_h$ is odd, then the right-hand side of~\eqref{eq:multi_moment_1} is equal to zero since $\sg(k,l)=0$ if $k+l$ is odd.
\end{rem}

We now turn to the proof of Lemma~\ref{lem:multi_moment}.
When $\min\{k_1,k_2,\ldots,k_h\}\le1$, Lemma~\ref{lem:multi_moment} is trivial because $\langle L_{H_n},1\rangle=1$ and  $\langle L_{H_n},x\rangle=0$ almost surely.
Hence, let $k_1,k_2,\ldots,k_h\ge2$ in what follows in this section.
In addition, we set $\bfk=(k_1,k_2,\ldots,k_h)$ and $|\bfk|\coloneqq k_1+k_2+\cdots+k_h$ for convenience.
Recall again that
\[
\langle L_{H_n},x^k\rangle =\frac1{d!\binom nd\{np(1-p)\}^{k/2}}\sum_{w=\widetilde{\sg_1}\sg_2\cdots\sg_{k+1}}T_n(w)
\]
for any $k\ge2$ as seen in~\eqref{eq:kth_moment}.
Here, $w=\widetilde{\sg_1}\sg_2\cdots\sg_{k+1}$ in the summation runs over all closed $(n,d)$-words of length $k+1$, and
\[
T_n(w)=\prod_{i=1}^k(A_{d-1}(Y^d_{n,p})-\E A_{d-1}(Y^d_{n,p}))_{\sg_i,\sg_{i+1}}.
\]
Therefore,
\begin{equation}\label{eq:multi_moment_2}
\E\Biggl[\prod_{j=1}^h(\langle L_{H_n},x^{k_j}\rangle-\E\langle L_{H_n},x^{k_j}\rangle)\Biggr]
=\frac1{\bigl\{d!\binom nd\bigr\}^h\{np(1-p)\}^{|\bfk|/2}}\sum_{a=(w^1,w^2,\ldots,w^h)}\bar T_n(a),
\end{equation}
where $a=(w^1,w^2,\ldots,w^h)$ in the summation runs over all $(n,d)$-sentences consisting of closed $(n,d)$-words $w^j$'s of length $k_j+1$, and
\begin{align}\label{eq:multi_moment_3}
\bar T_n(a)
&\coloneqq\E\Biggl[\prod_{j=1}^h\bigl(T_n(w^j)-\bar T_n(w^j)\bigr)\Biggr]\nonumber\\
&=\sum_{J\subset[h]}(-1)^{h-|J|}\E\Biggl[\prod_{j\in J}T_n(w^j)\Biggr]\Biggl(\prod_{j\in[h]\setminus J}\bar T_n(w^j)\Biggr)\nonumber\\
&=\sgn(a)\sum_{J\subset[h]}(-1)^{h-|J|}\Biggl(\prod_{\tau\in\bigcup_{j\in J}\Supp_d(w^j)}\E\Bigl[(\chi_p-p)^{\sum_{j\in J}N_{w^j}(\tau)}\Bigr]\Biggr)\nonumber\\
&\qad\cdot\Biggl(\prod_{j\in[h]\setminus J}\prod_{\tau\in\Supp_d(w^j)}\E\bigl[(\chi_p-p)^{N_{w^j}(\tau)}\bigr]\Biggr).
\end{align}
Note that $\bar T_n(a)=0$ unless for every $j\in[h]$, there exists $j'\in[h]\setminus\{j\}$ such that $\Supp_d(w^j)\cap\Supp_d(w^{j'})\neq\emptyset$.
Also, $\bar T_n(a)=0$ unless $N_a(\tau)\coloneqq\sum_{j=1}^hN_{w^j}(\tau)\ge2$ for all $\tau\in\Supp_d(a)$.
Furthermore, $\bar T_n(a)=\bar T_n(b)$ if $a$ and $b$ are equivalent $(n,d)$-sentences.
For $1\le s\le n$, let $\cW_{\bfk,s}^{(h)}$ denote the set of all representatives for the equivalence classes of sentences $a=(w^1,w^2,\ldots,w^h)$ consisting of closed $(n,d)$-words $w^j$'s of length $k_j+1$ such that for every $j\in[h]$, there exists $j'\in[h]\setminus\{j\}$ such that $\Supp_d(w^j)\cap\Supp_d(w^{j'})\neq\emptyset$, $N_a(\tau)\ge2$ for all $\tau\in\Supp_d(a)$, and $|\Supp_0(a)|=s$.
Then, by~\eqref{eq:multi_moment_2}, we obtain
\begin{equation}\label{eq:multi_moment_4}
\E\Biggl[\prod_{j=1}^h(\langle L_{H_n},x^{k_j}\rangle-\E\langle L_{H_n},x^{k_j}\rangle)\Biggr]
=\frac1{\bigl\{d!\binom nd\bigr\}^h\{np(1-p)\}^{|\bfk|/2}}\sum_{s=1}^n \sum_{a\in\cW_{\bfk,s}^{(h)}}|[a]|\bar T_n(a).
\end{equation}
Here, $[a]$ again denotes the equivalence class of the $(n,d)$-sentence $a$.
For a further calculation of~\eqref{eq:multi_moment_4}, the following lemma is crucial.
Recall that $\bar c(X)$ denotes the number of strongly connected components of a given pure $d$-dimensional simplicial complex $X$ (see Definition~\ref{df:pure}).
\begin{lem}\label{lem:cW_ksh}
Let $k_1,k_2,\ldots,k_h\ge2$ and $1\le s\le n$ be fixed, and let $a\in\cW_{\bfk,s}^{(h)}$.
Then,
\begin{equation}\label{eq:cW_ksh_1}
\bar c(X_a)\le\lfloor h/2\rfloor
\end{equation}
and
\begin{equation}\label{eq:cW_ksh_2}
d+1\le s\le\biggl\lfloor\frac{|\bfk|-h}2\biggr\rfloor+d\cdot\bar c(X_a).
\end{equation}
\end{lem}
The proof of Lemma~\ref{lem:cW_ksh} is deferred to the end of this section.
Combining Lemma~\ref{lem:cW_ksh} with~\eqref{eq:multi_moment_4}, we have
\begin{align}\label{eq:multi_moment_5}
&\E\Biggl[\prod_{j=1}^h(\langle L_{H_n},x^{k_j}\rangle-\E\langle L_{H_n},x^{k_j}\rangle)\Biggr]\nonumber\\
&=\sum_{s=d+1}^{\lfloor(|\bfk|-h)/2\rfloor+d\lfloor h/2\rfloor} \sum_{a\in\cW_{\bfk,s}^{(h)}}\frac{(n-d)(n-d-1)\cdots(n-s+1)}{\bigl\{d!\binom nd\bigr\}^{h-1}\{np(1-p)\}^{|\bfk|/2}}\bar T_n(a).
\end{align}
Here, we also used the fact that there exist exactly $n(n-1)\cdots(n-s+1)$ number of $(n,d)$-sentences that are equivalent to each $a\in\cW_{\bfk,s}^{(h)}$.
We omit the proof since it is given by the same argument as that for Lemma~\ref{lem:cW_ks2}(3).
By~\eqref{eq:multi_moment_3}, a similar calculation to~\eqref{eq:V_kth_moment_3} and~\eqref{eq:V_kth_moment_4} yields
\[
|\bar T_n(a)|
\le2^h\{p(1-p)\}^{|\Supp_d(a)|}
\le2^h\{p(1-p)\}^{s-d\cdot\bar c(X_a)}
\le2^h\{p(1-p)\}^{s-dh/2}
\]
for any $d+1\le s\le\lfloor(|\bfk|-h)/2\rfloor+d\lfloor h/2\rfloor$ and $a\in\cW_{\bfk,s}^{(h)}$.
Here, the second inequality follows from Lemma~\ref{lem:f0-fd}.
For the last inequality, we used~\eqref{eq:cW_ksh_1}.
We can also prove $|\cW_{\bfk,s}^{(h)}|\le (d!)^h(|\bfk|/2)^{d|\bfk|}$ in a similar way to the proof of Lemma~\ref{lem:cW_ks2}(2).
Therefore, for any $d+1\le s<(|\bfk|-h)/2+d(h/2)=(|\bfk|+dh-h)/2$,
\begin{align*}
\Biggl|\sum_{a\in\cW_{\bfk,s}^{(h)}}\frac{(n-d)(n-d-1)\cdots(n-s+1)}{\bigl\{d!\binom nd\bigr\}^{h-1}\{np(1-p)\}^{|\bfk|/2}}\bar T_n(a)\Biggr|
&\le\frac{C_{d,\bfk,h}}{n^{dh/2}\{np(1-p)\}^{(|\bfk|+dh)/2-s}}\\
&\le\frac{C_{d,\bfk,h}}{n^{dh/2}\{np(1-p)\}^{(h+1)/2}},
\end{align*}
where $C_{d,\bfk,h}$ is a constant depending only on $d$, $\bfk$, and $h$.
Therefore, by~\eqref{eq:multi_moment_5}, when $h$ is odd or $|\bfk|$ is odd,
\[
\lim_{n\to\infty}[n^d\{np(1-p)\}]^{h/2}\cdot\E\Biggl[\prod_{j=1}^h(\langle L_{H_n},x^{k_j}\rangle-\E\langle L_{H_n},x^{k_j}\rangle)\Biggr]
=0.
\]
Furthermore, when both $h$ and $|\bfk|$ are even, a straightforward calculation yields
\begin{align*}
&[n^d\{np(1-p)\}]^{h/2}\cdot\E\Biggl[\prod_{j=1}^h(\langle L_{H_n},x^{k_j}\rangle-\E\langle L_{H_n},x^{k_j}\rangle)\Biggr]\nonumber\\
&=(1+o_{d,h,\bfk}(1))\sum_{a\in\cW_{\bfk,(|\bfk|+dh-h)/2}^{(h)}}\frac{\bar T_n(a)}{\{p(1-p)\}^{(|\bfk|-h)/2}}+o_{d,h,\bfk}(1).
\end{align*}
Here, $o_{d,\bfk,h}(1)$ is a function of $n$, depending on $d$, $h$, and $\bfk$, that converges to zero as $n\to\infty$.
Now, we are ready to prove Lemma~\ref{lem:multi_moment}.
\begin{proof}[Proof of Lemma~\ref{lem:multi_moment}]
Combining the above discussion with Remark~\ref{rem:multi_moment}, it suffices to prove that whenever both $h$ and $|\bfk|$ are even,
\[
\lim_{n\to\infty}\sum_{a\in\cW_{\bfk,(|\bfk|+dh-h)/2}^{(h)}}\frac{\bar T_n(a)}{\{p(1-p)\}^{(|\bfk|-h)/2}}
=\sum_{\substack{\{j(c),j'(c)\}_{c=1}^{h/2}\colon\\\text{matching on $[h]$}}}\prod_{c=1}^{h/2}\sg(k_{j(c)},k_{j'(c)}).
\]
It follows from Lemma~\ref{lem:cW_ksh} that if $a\in\cW_{\bfk,(|\bfk|+dh-h)/2}^{(h)}$, then $X_a$ must have exactly $h/2$ strongly connected components.
Therefore, $a=(w^1,w^2,\ldots,w^h)\in\cW_{\bfk,(|\bfk|+dh-h)/2}^{(h)}$ determines a matching $\{j(c),j'(c)\}_{c=1}^{h/2}$ on $[h]$ such that every strongly connected component of $X_a$ is the support complex of the $(n,d)$-sentence $(w^{j(c)},w^{j'(c)})\in\cW_{k_{j(c)},k_{j'(c)},s_c}^{(2)}$ for some $1\le s_c\le n$.
Since $s_c\le\lfloor(k_{j(c)}+k_{j'(c)})/2\rfloor+d-1$ by Lemma~\ref{lem:cW_ks2}(1), it follows from $|\Supp_0(a)|=(|\bfk|+dh-h)/2$ that $s_c=(k_{j(c)}+k_{j'(c)})/2+d-1$ for all $c=1,2,\ldots,h/2$ and that every strongly connected component of $X_a$ is a connected component.
In particular, every $k_{j(c)}+k_{j'(c)}$ must be even.
Conversely, given a matching $\{j(c),j'(c)\}_{c=1}^{h/2}$ on $[h]$ such that every $k_{j(c)}+k_{j'(c)}$ is even and $a_c\in\cW_{k_{j(c)},k_{j'(c)},s_c}^{(2)}$ with $s_c\coloneqq(k_{j(c)}+k_{j'(c)})/2+d-1$~($c=1,2,\ldots,h/2$), we can recover $a\in\cW_{\bfk,(|\bfk|+dh-h)/2}^{(h)}$.
Furthermore, $\bar T_n(a)=\prod_{c=1}^{h/2}\bar T_n(a_c)$ in this case.
Thus,
\begin{align}\label{eq:multi_moment_6}
&\sum_{a\in\cW_{\bfk,(|\bfk|+dh-h)/2}^{(h)}}\frac{\bar T_n(a)}{\{p(1-p)\}^{(|\bfk|-h)/2}}\nonumber\\
&=\sum_{\{j(c),j'(c)\}_{c=1}^{h/2}}\sum_{\substack{a_1,a_2,\ldots,a_{h/2}\colon\\a_c\in\cW_{k_{j(c)},k_{j'(c)},s_c}^{(2)}}}\prod_{c=1}^{h/2}\frac{\bar T_n(a_c)}{\{p(1-p)\}^{(k_{j(c)}+k_{j'(c)})/2-1}}\nonumber\\
&=\sum_{\{j(c),j'(c)\}_{c=1}^{h/2}}\prod_{c=1}^{h/2}\Biggl(\sum_{a_c\in\cW_{k_{j(c)},k_{j'(c)},s_c}^{(2)}}\frac{\bar T_n(a_c)}{\{p(1-p)\}^{(k_{j(c)}+k_{j'(c)})/2-1}}\Biggr)\nonumber\\
&\xrightarrow[n\to\infty]{}\sum_{\{j(c),j'(c)\}_{c=1}^{h/2}}\prod_{c=1}^{h/2}\sg(k_{j(c)},k_{j'(c)})\nonumber\\
&=\sum_{\substack{\{j(c),j'(c)\}_{c=1}^{h/2}\colon\\\text{matching on $[h]$}}}\prod_{c=1}^{h/2}\sg(k_{j(c)},k_{j'(c)}).
\end{align}
Here, each $\{j(c),j'(c)\}_{c=1}^{h/2}$ in the summations in the second, third, and fourth lines runs over all matchings on $[h]$ such that $k_{j(c)}+k_{j'(c)}$ is even for any $c=1,2,\ldots,h/2$.
For the fourth line of~\eqref{eq:multi_moment_6}, we note that
\[
\lim_{n\to\infty}\sum_{a_i\in\cW_{k_{j(c)},k_{j'(c)},s_c}^{(2)}}\frac{\bar T_n(a_i)}{\{p(1-p)\}^{(k_{j(c)}+k_{j'(c)})/2-1}}
=\sg(k_{j(c)},k_{j'(c)})
\]
as seen in the proof of Theorem~\ref{thm:multi_CLT}(1).
The final line of~\eqref{eq:multi_moment_6} follows from the fact that if $k_{j(c)}+k_{j'(c)}$ is odd, then $\sg(k_{j(c)},k_{j'(c)})=0$ by the definition.
\end{proof}

Lastly, we prove Lemma~\ref{lem:cW_ksh}, where the notion of $d$-forests is crucial.
\begin{df}[$d$-forest]
We call a union of finitely many $d$-trees $X_1,X_2,\ldots,X_c$ a \textit{$d$-forest} if $\dim(X_i\cap X_{i'})<d-1$ for any $i\neq i'\in[c]$.
\end{df}
We use the following lemma in the proof of Lemma~\ref{lem:cW_ksh}.
\begin{lem}\label{lem:d-forest}
Let $X$ be a strongly connected pure $d$-dimensional simplicial complex.
Then, there exists a $d$-forest $X'\subset X$ such that $F_0(X')=F_0(X)$ and $f_0(X')=f_d(X')+d$.
\end{lem}
\begin{proof}
We modify the generating process~\eqref{eq:process} by not adding the $d$-simplex $\sg_{t-1}\cup\{v_t\}$ if $v_t\notin V(X(t-1))$ at every step $t=1,2,\ldots,f_d(X)$.
We denote the resulting subcomplex of $X$ by $X'$.
In other words, letting $T$ denote the set of all indices $t\in\{1,2,\ldots,f_d(X)\}$ such that $v_t\notin V(X(t-1))$, define $X'\coloneqq X\setminus\{\sg_{t-1}\cup\{v_t\}\}_{t\in T}$.
Then, obviously, $X'$ is a $d$-forest and $F_0(X')=F_0(X)$.
Furthermore, by tracking the change of the numbers of vertices and $d$-simplices in the modified process, we obtain $f_0(X')=f_d(X')+d$.
\end{proof}
We prove Lemma~\ref{lem:cW_ksh} using Lemma~\ref{lem:d-forest}.
\begin{proof}[Proof of Lemma~\ref{lem:cW_ksh}]
Let $a=(w^1,w^2,\ldots,w^h)\in\cW_{\bfk,s}^{(h)}$.
Since the support complex of each closed $(n,d)$-word in $a$ has a $d$-simplex in common with at least one support complex of another $(n,d)$-word, $\bar c(X_a)\le\lfloor h/2\rfloor$ immediately follows.
The first inequality in~\eqref{eq:cW_ksh_2} is also trivial since $k_1,k_2,\ldots,k_h\ge2$.
For the second inequality in~\eqref{eq:cW_ksh_2}, let $X_1,X_2,\ldots,X_{\bar c(X_a)}$ be the strongly connected components of $X_a$.
By Lemma~\ref{lem:d-forest}, for each strongly connected component $X_c$~($c=1,2,\ldots,\bar c(X_a)$), we can find a $d$-forest $X_c'\subset X_c$ such that $F_0(X_c')=F_0(X_c)$ and $f_0(X_c')=f_d(X_c')+d$.
We define a $d$-forest $X_a'$ by
\[
X_a'\coloneqq\bigcup_{c=1}^{\bar c(X_a)}X_c'.
\]
Then, we have
\begin{equation}\label{eq:cW_ksh_3}
s=|\Supp_0(a)|
=f_0(X_a)
\le\sum_{c=1}^{\bar c(X_a)}f_0(X_c)
=\sum_{c=1}^{\bar c(X_a)}(f_d(X_c')+d)
=f_d(X_a')+d\cdot\bar c(X_a).
\end{equation}

Now, we introduce a \textit{$d$-simplex-bounding table}, which is a higher-dimensional analog of the \textit{edge-bounding table}~\cite[Chapter~2.1.7]{AGZ11}.
Set $w^j=\widetilde\sg^j_1\sg^j_2\cdots\sg^j_{k_j+1}$ for all $j\in[h]$, and
\[
I\coloneqq\bigcup_{j=1}^h\{j\}\times[k_j].
\]
Let $B=\{b^j_i\}_{(j,i)\in I}$ be a table whose all entries are either zero or one.
We call $B$ a $d$-simplex-bounding table if the following conditions are satisfied:
\begin{enumerate}
\item for all $(j,i)\in I$, if $b^j_i=1$, then $\sg^j_i\cup\sg^j_{i+1}\in F_d(X_a')$;
\item for each $\tau\in F_d(X_a')$, there exist distinct $(j_1,i_1),(j_2,i_2)\in I$ such that $\sg^{j_1}_{i_1}\cup\sg^{j_1}_{i_1+1}=\sg^{j_2}_{i_2}\cup\sg^{j_2}_{i_2+1}=\tau$ and $b^{j_1}_{i_1}=b^{j_2}_{i_2}=1$;
\item for each $\tau\in F_d(X_a')$ and $j\in[h]$, if $\tau\in\Supp_d(w^j)$, then there exists $(j,i)\in I$ such that $\sg^j_i\cup\sg^j_{i+1}=\tau$ and $b^j_i=1$.
\end{enumerate}
At least one $d$-simplex-bounding table exists.
Indeed, if we set $b^j_i=1$ for each $(j,i)\in I$ such that $\sg^j_i\cup\sg^j_{i+1}\in F_d(X_a')$ and $b^j_i=0$ elsewhere, then $B=\{b^j_i\}_{(j,i)\in I}$ is a $d$-simplex-bounding table.
Furthermore, for any $d$-simplex-bounding table $B$,
\begin{equation}\label{eq:bounding_table}
f_d(X_a')\le\frac12\sum_{(j,i)\in I}b^j_i
\end{equation}
because of Condition~(2).
Now, let $B=\{b^j_i\}_{(j,i)\in I}$ be a $d$-simplex-bounding table whose all entries in the $j_0$th row are equal to $1$.
Then, $X_{w^{j_0}}\subset X_a'$ because of Condition~(1).
Noting that $w^{j_0}$ is a closed $(n,d)$-word and $X_a'$ is a $d$-forest, we have $N_{w^{j_0}}(\tau)\ge2$ for any $\tau\in\Supp_d(w^{j_0})$.
From the definition of $\cW_{\bfk,s}^{(h)}$, we can find $(j_0,i_0)\in I$ such that $\sg^{j_0}_{i_0}\cup\sg^{j_0}_{i_0+1}$ appears in the $d$-support of another closed $(n,d)$-word different from $w^{j_0}$.
We then define a new table obtained by replacing the entry one in position $(j_0,i_0)\in I$ by the entry zero.
It is easy to verify that the new table is again a $d$-simplex-bounding table.
Repeating this modification, we eventually find a $d$-simplex-bounding table with zero appearing at least once in every row, which together with~\eqref{eq:bounding_table} implies that
\[
f_d(X_a')
\le\Biggl\lfloor\frac12\sum_{j=1}^h(k_j-1)\Biggr\rfloor
\le\biggl\lfloor\frac{|\bfk|-h}2\biggr\rfloor.
\]
Combining this with~\eqref{eq:cW_ksh_3}, we obtain the second inequality in~\eqref{eq:cW_ksh_2}.
\end{proof}

\section{Central limit theorem for differentiable test functions}\label{sec:CLT_diff_test_funct}
Our goal in this section is to prove a CLT for $\langle L_{H_n},f\rangle$, where $f\colon\R\to\R$ of polynomial growth that is of class $C^2$ on a closed interval.
To this end, the notions of the polynomial-type CLT and the convex concentration property for a sequence of random symmetric matrices are crucial.
In Subsection~\ref{ssec:conc-polyCLT}, we introduce the definitions of these notions.
Furthermore, we state a general theorem that lifts a polynomial-type CLT to a CLT for such $C^2$ test functions.
In Subsection~\ref{ssec:Talagrand}, we provide an estimate of the variance of $\langle L_{H_n},f\rangle$, which is the first condition for the convex concentration property of $H_n$, using Talagrand's concentration inequality.
In Subsection~\ref{ssec:pf_main_thm}, we prove Theorem~\ref{thm:CLT_diff}.

\subsection{Polynomial-type CLT and convex concentration property}\label{ssec:conc-polyCLT}
Throughout this subsection, let $\{Y_n\}_n$ be a sequence of random symmetric matrices, and let $\{s_n\}_n$ be a sequence of positive numbers.
\begin{df}[cf.~{\cite[Definition~11.5]{AZ06}}]\label{df:poly-type_CLT}
$\{Y_n\}_n$ is said to satisfy the \textit{polynomial-type CLT} with normalization $s_n$ if there exists a mean-zero Gaussian process $\{W_k\}_{k=0}^\infty$ such that for any real-valued polynomial function $f(x)=\sum_{k=0}^Ka_kx^k$,
\[
\frac{\langle L_{Y_n},f\rangle-\E\langle L_{Y_n},f\rangle}{s_n}
\xrightarrow[n\to\infty]{d}W_f\coloneqq\sum_{k=0}^Ka_kW_k,
\]
and it holds that
\[
\lim_{n\to\infty}\Var\biggl(\frac{\langle L_{Y_n},f\rangle}{s_n}\biggr)=\E[W_f^2].
\]
\end{df}
\begin{rem}\label{rem:poly-type_CLT}
As seen in~\eqref{eq:multi_moment_00} and~\eqref{eq:multi_moment_000}, the sequence $\{H_n\}_{n>d}$ satisfies a polynomial-type CLT with normalization $s_n=(n^d\{np(1-p)\})^{-1/2}$.
\end{rem}
Next, we introduce the notion of the convex concentration property.
For any Lipschitz function $g\colon\R^N\to\R$, its Lipschitz constant is defined by
\[
\Lip(g)\coloneqq\sup_{x\neq y\in\R^N}\frac{|g(x)-g(y)|}{\|x-y\|_{\R^N}},
\]
where $\|\cdot\|_{\R^N}$ is the Euclidean norm in $\R^N$.
Note that if $g\colon\R\to\R$ is a Lipschitz differentiable function, then $\Lip(g)=\sup_{x\in\R}|g'(x)|$.
Recall that a function $g\colon\R\to\R$ is \textit{of polynomial growth} if there exists a constant $M\ge0$ such that $\lim_{x\to\pm\infty}|g(x)|/|x|^M=0$.
\begin{df}\label{df:concentration}
$\{Y_n\}_{n\in\N}$ is said to satisfy the \textit{convex concentration property} with normalization $s_n$ if the following two conditions are satisfied.
\begin{enumerate}
\item There exists a constant $c>0$ such that for any convex Lipschitz function $g\colon\R\to\R$,
\begin{equation}\label{eq:concentration_1}
\sup_n\Var\biggl(\frac{\langle L_{Y_n},g\rangle}{s_n}\biggr)
\le c\Lip(g)^2.
\end{equation}
\item There exists a constant $M\ge0$ such that for any measurable function $g\colon\R\to\R$ of polynomial growth with $\Supp(g)\subset[-M,M]^c$, it holds that
\begin{equation}\label{eq:concentration_2}
\lim_{n\to\infty}\E\biggl[\biggl(\frac{\langle L_{Y_n},g\rangle}{s_n}\biggr)^2\biggr]=0.
\end{equation}
\end{enumerate}
\end{df}
\begin{rem}
In the definition of the concentration property introduced in~\cite[Definition~11.2]{AZ06}, the convexity of $g$ in~\eqref{eq:concentration_1} is not assumed.
In this sense, the convex concentration property is a slightly weaker version of the concentration property in~\cite{AZ06}.
\end{rem}
Our aim in this subsection is to establish a general theorem that lifts a polynomial-type CLT to a CLT for a larger class of test functions through the convex concentration property.
To this end, we need several lemmas.
In what follows, we write $\|f-g\|_S=\sup_{x\in S}|f(x)-g(x)|$ for any function $f,g\colon\R\to\R$ and $S\subset\R$.
\begin{lem}\label{lem:SW}
Let $a<b$, and let $f\colon\R\to\R$ be a function that is continuous on $[a,b]$.
Then, there exists a sequence $(R_m)_{m\in\N}$ of real-valued polynomial functions such that $\lim_{m\to\infty}\|f-R_m\|_{[a,b]}=0$ and that $R_m(x)<f(x)$ for any $m\in\N$ and $x\in[a,b]$.
\end{lem}
\begin{proof}
Let $\eps>0$.
It suffices to prove that there exists a real-valued polynomial function $R$ such that $\|f-R\|_{[a,b]}<\eps$ and that $R(x)<f(x)$ for any $x\in[a,b]$.
By the Stone--Weierstrass theorem, we can take a real-valued polynomial function $\widetilde R$ such that $\|f-\widetilde R\|_{[a,b]}<\eps/2$.
Define a real-valued polynomial function $R$ by $R(x)\coloneqq \widetilde R(x)-\eps/2$ for any $x\in\R$. 
Then, $\|f-R\|_{[a,b]}\le\|f-\widetilde R\|_{[a,b]}+\eps/2<\eps$.
Furthermore, $R(x)-f(x)=\widetilde R(x)-f(x)-\eps/2\le\|f-\widetilde R\|_{[a,b]}-\eps/2<0$ for any $x\in[a,b]$.
\end{proof}
\begin{lem}\label{lem:poly_approx}
Let $a<b$, and let $f\colon\R\to\R$ be a function that is of class $C^2$ on $[a,b]$.
Then, there exists a sequence $(P_m)_{m\in\N}$ of real-valued polynomial functions such that $\lim_{m\to\infty}\|f'-P'_m\|_{[a,b]}=0$ and that $f-P_m$ is strictly convex on $[a,b]$.
\end{lem}
\begin{proof}
Since $f''$ is continuous on $[a,b]$, we can take a sequence $(R_m)_{m\in\N}$ of real-valued polynomial functions such that $\lim_{m\to\infty}\|f''-R_m\|_{[a,b]}=0$ and that $R_m(x)<f''(x)$ for any $m\in\N$ and $x\in[a,b]$ by Lemma~\ref{lem:SW}.
Set $a<c<b$, and define
\[
Q_m(x)\coloneqq f'(c)+\int_c^xR_m(y)\,dy
\quad\text{and}\quad
P_m(x)\coloneqq\int_c^xQ_m(y)\,dy.
\]
Then, noting that $f'(x)=f'(c)+\int_c^xf''(y)\,dy$, we have
\begin{align*}
\|f'-P'_m\|_{[a,b]}
&=\sup_{x\in[a,b]}|f'(x)-Q_m(x)|\\
&=\sup_{x\in[a,b]}\biggl|\int_c^xf''(y)-R_m(y)\,dy\biggr|\\
&\le(b-a)\|f''-R_m\|_{[a,b]}
\xrightarrow[m\to\infty]{}0.
\end{align*}
Furthermore, $(f-P_m)''(x)=f''(x)-R_m(x)>0$ for any $x\in[a,b]$, which implies that $f-P_m$ is strictly convex on $[a,b]$.
\end{proof}
We also use the following approximation approach, which is often useful to obtain a CLT from a collection of relatively simple CLTs.
\begin{lem}[{\cite[Lemma~2.2]{Tr19}}]\label{lem:approx_CLT}
Let $\{\xi_n\}_{n\in\N}$ and $\{\xi_{n,m}\}_{n,m\in\N}$ be families of mean-zero random variables.
Assume that
\begin{enumerate}
\item for each $m\in\N$, the limit $\sg_m^2\coloneqq\lim_{n\to\infty}\Var(\xi_{n,m})\in[0,\infty)$ exists and $\xi_{n,m}\xrightarrow[n\to\infty]{d}\cN(0,\sg_m^2)$,
\item it holds that $\lim_{m\to\infty}\limsup_{n\to\infty}\Var(\xi_n-\xi_{n,m})=0$.
\end{enumerate}
Then, the limit $\sg^2\coloneqq\lim_{m\to\infty}\sg_m^2\in[0,\infty)$ exists, $\lim_{n\to\infty}\Var(\xi_n)=\sg^2$, and $\xi_n\xrightarrow[n\to\infty]{d}\cN(0,\sg^2)$.
\end{lem}
The following theorem provides a sufficient condition under which one can lift a polynomial-type CLT to a CLT for a larger class of test functions.
The proof follows from Lemmas~\ref{lem:poly_approx} and~\ref{lem:approx_CLT} involving the Stone--Weierstrass theorem.
\begin{thm}\label{thm:CLT_poly-diff}
Assume that $\{Y_n\}_n$ satisfies both a polynomial-type CLT and the convex concentration property with normalization $s_n$.
Let $M\ge0$ be the constant in Definition~\ref{df:concentration}.
Then, for any function $f\colon\R\to\R$ of polynomial growth that is of class $C^2$ on $[-M,M]$, there exists a constant $\sg_f^2\ge0$ such that
\[
\lim_{n\to\infty}\Var\biggl(\frac{\langle L_{Y_n},f\rangle)}{s_n}\biggr)
=\sg_f^2
\]
and it holds that
\[
\frac{\langle L_{Y_n},f\rangle-\E\langle L_{Y_n},f\rangle}{s_n}
\xrightarrow[n\to\infty]{d}\cN(0,\sg_f^2).
\]
\end{thm}
\begin{proof}
Let $f\colon\R\to\R$ be a function of polynomial growth that is of class $C^2$ on $[-M,M]$.
We can take $M'>M$ such that $f$ is of class $C^2$ on $[-M',M']$.
By Lemma~\ref{lem:poly_approx}, there exists a sequence $(P_m)_{m\in\N}$ of real-valued polynomial functions such that $\lim_{m\to\infty}\|f'-P'_m\|_{[-M',M']}=0$ and that $f-P_m$ is convex on $[-M',M']$.
Now, we define
\[
\widetilde f(x)\coloneqq\begin{cases}
f(-M')+f'(-M')(x+M')    &\text{if $x<-M'$,}\\
f(x)    &\text{if $-M'\le x\le M'$,}\\
f(M')+f'(M')(x-M')    &\text{if $x>M'$}
\end{cases}
\]
and
\[
\widetilde P_m(x)\coloneqq\begin{cases}
P_m(-M')+P'_m(-M')(x+M')    &\text{if $x<-M'$,}\\
P_m(x)    &\text{if $-M'\le x\le M'$,}\\
P_m(M')+P'_m(M')(x-M')    &\text{if $x>M'$.}
\end{cases}
\]
Since $f-P_m$ is convex on $[-M',M']$, it is easy to verify that $\widetilde f-\widetilde P_m$ is a convex function.
Indeed, $(\widetilde f-\widetilde P_m)'$ is nondecreasing.
We also note that
\begin{equation}\label{eq:CLT_poly-diff_1}
\Lip(\widetilde f-\widetilde P_m)
=\|(\widetilde f-\widetilde P_m)'\|_\R
=\|(f-P_m)'\|_{[-M',M']}.
\end{equation}
Now, we set
\[
\eta_n=\frac{\langle L_{Y_n},f\rangle-\E\langle L_{Y_n},f\rangle}{s_n},
\quad
\eta_{n,m}=\frac{\langle L_{Y_n},P_m\rangle-\E\langle L_{Y_n},P_m\rangle}{s_n}
\]
and
\[
\xi_n=\frac{\langle L_{Y_n},\widetilde f\rangle-\E\langle L_{Y_n},\widetilde f\rangle}{s_n},
\quad
\xi_{n,m}=\frac{\langle L_{Y_n},\widetilde P_m\rangle-\E\langle L_{Y_n},\widetilde P_m\rangle}{s_n}.
\]
By the polynomial-type CLT, there exists a constant $\sg_m^2\ge0$ such that
\begin{equation}\label{eq:CLT_poly-diff_2}
\eta_{n,m}\xrightarrow[n\to\infty]{d}\cN(0,\sg_m^2)
\quad\text{and}\quad
\lim_{n\to\infty}\Var\biggl(\frac{\langle L_{Y_n},P_m\rangle}{s_n}\biggr)=\sg_m^2.
\end{equation}
Since $\Supp(\widetilde P_m-P_m)\subset[-M,M]^c$,~\eqref{eq:concentration_2} yields
\begin{align*}
\|\xi_{n,m}-\eta_{n,m}\|_{L^2}^2
=\E[(\xi_{n,m}-\eta_{n,m})^2]
&=\Var\biggl(\frac{\langle L_{Y_n},\widetilde P_m-P_m\rangle}{s_n}\biggr)\\
&\le\E\biggl[\biggl(\frac{\langle L_{Y_n},\widetilde P_m-P_m\rangle}{s_n}\biggr)^2\biggr]
\xrightarrow[n\to\infty]{}0.
\end{align*}
In particular, $\xi_{n,m}-\eta_{n,m}\xrightarrow[n\to\infty]{d}0$.
Therefore,
\begin{equation}\label{eq:CLT_poly-diff_3}
\xi_{n,m}=\eta_{n,m}+(\xi_{n,m}-\eta_{n,m})\xrightarrow[n\to\infty]{d}\cN(0,\sg_m^2).
\end{equation}
Furthermore, noting that
\[
\biggl|\sqrt{\Var\biggl(\frac{\langle L_{Y_n},\widetilde P_m\rangle}{s_n}\biggr)}-\sqrt{\Var\biggl(\frac{\langle L_{Y_n},P_m\rangle}{s_n}\biggr)}\biggr|=|\|\xi_{n,m}\|_{L^2}-\|\eta_{n,m}\|_{L^2}|
\le\|\xi_{n,m}-\eta_{n,m}\|_{L^2},
\]
we have
\begin{equation}\label{eq:CLT_poly-diff_4}
\lim_{n\to\infty}\Var\biggl(\frac{\langle L_{Y_n},\widetilde P_m\rangle}{s_n}\biggr)=\sg_m^2.
\end{equation}
Let $c\ge0$ be the constant in Definition~\ref{df:concentration}.
Since $\widetilde f-\widetilde P_m$ is a convex Lipschitz function, it follows from~\eqref{eq:concentration_1} and~\eqref{eq:CLT_poly-diff_1} that
\[
\limsup_{n\to\infty}\Var\biggl(\frac{\langle L_{Y_n},\widetilde f-\widetilde P_m\rangle}{s_n}\biggr)
\le c\|(f-P_m)'\|_{[-M',M']}^2.
\]
The right-hand side converges to zero as $m\to\infty$.
Therefore, Lemma~\ref{lem:approx_CLT} implies that the limit $\sg^2\coloneqq\lim_{m\to\infty}\sg_m^2\in[0,\infty)$ exists and that
\[
\xi_n\xrightarrow[n\to\infty]{d}\cN(0,\sg^2)
\quad\text{and}\quad
\lim_{n\to\infty}\Var\biggl(\frac{\langle L_{Y_n},\widetilde f\rangle}{s_n}\biggr)=\sg^2.
\]
Noting that $\Supp(\widetilde f-f)\subset[-M,M]^c$, we have
\[
\eta_n\xrightarrow[n\to\infty]{d}\cN(0,\sg^2)
\quad\text{and}\quad
\lim_{n\to\infty}\Var\biggl(\frac{\langle L_{Y_n},f\rangle}{s_n}\biggr)=\sg^2
\]
in a similar way to obtain both~\eqref{eq:CLT_poly-diff_3} and~\eqref{eq:CLT_poly-diff_4} from~\eqref{eq:CLT_poly-diff_2}.
\end{proof}

\subsection{Talagrand's concentration inequality}\label{ssec:Talagrand}
In this subsection, we prove that the sequence $\{H_n\}_{n>d}$ satisfies Condition~(1) in Definition~\ref{df:concentration} with normalization $s_n=(n^d\{np(1-p)\})^{-1/2}$.
A key tool is Talagrand's concentration inequality~\cite{Tal95}, which allows us to observe that the convex Lipschitz function of independent random variables are not much random.
\begin{thm}[{\cite[Theorem~2.1.13]{Tao12}}]\label{thm:TCI}
Let $K>0$, and let $\{X_i\}_{i=1}^N$ be independent real-valued random variables with $|X_i|\le K$ for all $i=1,2,\ldots,N$.
Let $f\colon\R^N\to\R$ be a convex Lipschitz function with $\Lip(f)\le1$.
Then, there exist constants $c_1,c_2>0$ such that for every $\lm\ge0$,
\[
\P(|f(X_1,X_2,\ldots,X_N)-\M f(X_1,X_2,\ldots,X_N)|\ge\lm K)
\le c_1\exp(-c_2\lm^2).
\]
Here, $\M f(X_1,X_2,\ldots,X_N)$ is a median of $f(X_1,X_2,\ldots,X_N)$.
\end{thm}
See also~\cite[Section~1.5]{Tao11} for the case where $X_i$'s are Bernoulli random variables.
The following is an easy application of Theorem~\ref{thm:TCI}.
\begin{cor}\label{cor:TCI}
Let $K>0$, and let $\{X_i\}_{i=1}^N$ be independent real-valued random variables with $|X_i|\le K$ for all $i=1,2,\ldots,N$.
Let $F\colon\R^N\to\R$ be a convex Lipschitz function.
Then, there exists a constant $C>0$ such that
\[
\Var(F(X_1,X_2,\ldots,X_N))\le CK^2\Lip(F)^2.
\]
\end{cor}
\begin{proof}
When $\Lip(F)=0$, the conclusion is trivial.
Hence, we assume $\Lip(F)>0$ and set $f=F/\Lip(F)$.
Noting that $\Lip(f)\le1$, we take $c_1,c_2>0$ as in Theorem~\ref{thm:TCI}.
Then,
\begin{align*}
\Var(F(X_1,X_2,\ldots,X_N))/\Lip(F)^2
&=\E[\{f(X_1,X_2,\ldots,X_N)-\E f(X_1,X_2,\ldots,X_N)\}^2]\\
&\le\E[\{f(X_1,X_2,\ldots,X_N)-\M f(X_1,X_2,\ldots,X_N)\}^2]\\
&=\int_0^\infty\P(\{f(X_1,X_2,\ldots,X_N)-\M f(X_1,X_2,\ldots,X_N)\}^2\ge t)\,dt\\
&=\int_0^\infty\P(|f(X_1,X_2,\ldots,X_N)-\M f(X_1,X_2,\ldots,X_N)|\ge\sqrt t)\,dt\\
&\le c_1\int_0^\infty\exp(-(c_2/K^2)t)\,dt\\
&=(c_1/c_2)K^2,
\end{align*}
which completes the proof.
\end{proof}

Now, we estimate $\Var(\langle L_{H_n},g\rangle)$ for a given convex Lipschitz function $g\colon\R\to\R$ using Corollary~\ref{cor:TCI}.
For convenience, we assume the Linial--Meshulam complex $Y^d_{n,p}$ is constructed from independent Bernoulli random variables $\{b_{n,\tau}\}_{\tau\in F_d(\cK_n)}$ with parameter $p$:
\[
Y^d_{n,p}=\cK_n^{(d-1)}\cup\{\tau\in F_d(\cK_n)\mid b_{n,\tau}=1\}.
\]
We introduce a function $F_g\colon\R^{F_d(\cK_n)}\to\R$ satisfying that
\begin{equation}\label{eq:Fg}
\langle L_{H_n},g\rangle
=F_g\Bigl(\{b_{n,\tau}\}_{\tau\in F_d(\cK_n)}\Bigr)
\end{equation}
as follows.
Let $\Sym$ be the set of all $F_{d-1}(\cK_n)\times F_{d-1}(\cK_n)$ real symmetric matrices, equipped with a norm defined by
\[
\|H\|\coloneqq\sqrt{\Tr(H^2)}
=\biggl(\sum_{i,j}H_{i,j}^2\biggr)^{1/2}
\]
for any $H\in\Sym$.
We define a map $F_1\colon\R^{F_d(\cK_n)}\to\Sym$ by
\[
\Bigl(F_1\Bigl(\{b_\tau\}_{\tau\in F_d(\cK_n)}\Bigr)\Bigr)_{\sg,\sg'}\coloneqq\begin{cases}
\sgn(\sg,\sg')b_{\sg\cup\sg'}	&\text{if $\sg\cup\sg'\in F_d(\cK_n)$,}\\
0			&\text{otherwise}
\end{cases}
\]
for any $\{b_\tau\}_{\tau\in F_d(\cK_n)}\in\R^{F_d(\cK_n)}$ and $\sg,\sg'\in F_{d-1}(\cK_n)$.
Note that $F_1\Bigl(\{b_{n,\tau}\}_{\tau\in F_d(\cK_n)}\Bigr)=A_{d-1}(Y^d_{n,p})$.
We also define a map $F_2\colon\Sym\to\Sym$ by
\[
F_2(A)\coloneqq\frac1{\sqrt{np(1-p)}}(A-\E[A_{d-1}(Y^d_{n,p})])
\]
for any $A\in\Sym$.
Then, $(F_2\circ F_1)\Bigl(\{b_{n,\tau}\}_{\tau\in F_d(\cK_n)}\Bigr)=F_2\bigl(A_{d-1}(Y^d_{n,p})\bigr)=H_n$.
Next, we define a function $F_3\colon\Sym\to\R^{\binom nd}$ by
\[
F_3(H)\coloneqq\Bigl(\lm_1[H],\lm_2[H],\ldots,\lm_{\binom nd}[H]\Bigr)
\]
for any $H\in\Sym$.
Here, $\lm_1[H]\ge\lm_2[H]\ge\ldots\ge\lm_{\binom nd}[H]$ are the real eigenvalues of $H$.
Furthermore, we define a function $F_4\colon\R^{\binom nd}\to\R$ by
\[
F_4\Bigl(\lm_1,\lm_2,\ldots,\lm_{\binom nd}\Bigr)\coloneqq\frac1{\binom nd}\sum_{i=1}^{\binom nd}g(\lm_i)
\]
for any $\Bigl(\lm_1,\lm_2,\ldots,\lm_{\binom nd}\Bigr)\in\R^{\binom nd}$.
Obviously, $F_g\coloneqq F_4\circ F_3\circ F_2\circ F_1$ satisfies~\eqref{eq:Fg}.
\begin{lem}\label{lem:Lip_const}
Let $g\colon\R\to\R$ be a convex Lipschitz function.
Then, the function $F_g\colon\R^{F_d(\cK_n)}\to\R$ defined above is convex Lipschitz, and it holds that
\[
\Lip(F_g)
\le\sqrt{\frac{(d+1)d}{\binom nd\{np(1-p)\}}}\Lip(g).
\]
\end{lem}
\begin{proof}
Since
\begin{align*}
\Bigl\|F_1\Bigl(\{b_\tau\}_{\tau\in F_d(\cK_n)}\Bigr)-F_1\Bigl(\{b'_\tau\}_{\tau\in F_d(\cK_n)}\Bigr)\Bigr\|^2
&=\sum_{\tau\in F_d(\cK_n)}\sum_{\substack{\sg,\sg'\in F_{d-1}(\cK_n)\\\sg\cup\sg'=\tau}}(b_\tau-b'_\tau)^2\\
&=(d+1)d\sum_{\tau\in F_d(\cK_n)}(b_\tau-b'_\tau)^2\\
&=(d+1)d\|\{b_\tau\}_{\tau\in F_d(\cK_n)}-\{b'_\tau\}_{\tau\in F_d(\cK_n)}\|_{\R^{F_d(\cK_n)}}^2
\end{align*}
for any $\{b_\tau\}_{\tau\in F_d(\cK_n)},\{b'_\tau\}_{\tau\in F_d(\cK_n)}\in\R^{F_d(\cK_n)}$, we have $\Lip(F_1)=\sqrt{(d+1)d}$.
Obviously, $\Lip(F_2)=\{np(1-p)\}^{-1/2}$.
By the Hoffman--Wielandt inequality~\cite[Lemma~2.1.19]{AGZ11}, we have $\Lip(F_3)\le1$.
Furthermore, we can easily verify that
\[
\Lip(F_4)\le\binom nd^{-1/2}\Lip(g).
\]
Therefore,
\[
\Lip(F_g)
\le \Lip(F_4)\Lip(F_3)\Lip(F_2)\Lip(F_1)
\le\sqrt{\frac{(d+1)d}{\binom nd\{np(1-p)\}}}\Lip(g).
\]

Furthermore, by Klein's lemma (see, e.g.,~\cite[Lemma 4.4.12]{AGZ11}), the function
\[
\Sym\ni H\mapsto\sum_{i=1}^{\binom nd}g(\lm_i[H])\in\R
\]
is convex.
Therefore, noting that $F_1$ is a linear function, we can easily verify that $F_g=F_4\circ F_3\circ F_2\circ F_1$ is also convex.
\end{proof}

Combining Corollary~\ref{cor:TCI} with Lemma~\ref{lem:Lip_const}, we can prove that the sequence $\{H_n\}_{n>d}$ satisfies Condition~(1) in Definition~\ref{df:concentration} with normalization $s_n=(n^d\{np(1-p)\})^{-1/2}$ as follows.
\begin{lem}\label{lem:cons_condi1}
There exists a constant $c>0$ such that for any convex Lipschitz function $g\colon\R\to\R$,
\[
\sup_{n>d}n^d\{np(1-p)\}\Var(\langle L_{H_n},g\rangle)
\le c\Lip(g)^2.
\]
\end{lem}
\begin{proof}
We take a constant $C>0$ as in Corollary~\ref{cor:TCI} with $K=1$, independent real-valued random variables $\{b_{n,\tau}\}_{\tau\in F_d(\cK_n)}$, and $F=F_g$.
Then, for any $n>d$,
\begin{align*}
n^d\{np(1-p)\}\Var(\langle L_{H_n},g\rangle)
&=n^d\{np(1-p)\}\Var\bigl(F_g\bigl(\{b_{n,\tau}\}_{\tau\in F_d(\cK_n)}\bigr)\bigr)\\
&\le Cn^d\{np(1-p)\}\Lip(F_g)^2\\
&\le C(d+1)d\frac{n^d}{\binom nd}\Lip(g)^2\\
&\le C(d+1)^dd\Lip(g)^2.
\end{align*}
The third line follows from Lemma~\ref{lem:Lip_const}.
For the fourth line, we note that $n^d/\binom nd\le(d+1)^{d-1}$ for any $n>d$.
\end{proof}

\subsection{Proof of Theorem~\ref{thm:CLT_diff}}\label{ssec:pf_main_thm}
In this subsection, we prove Theorem~\ref{thm:CLT_diff}.
As mentioned in Remark~\ref{df:poly-type_CLT}, the sequence $\{H_n\}_{n>d}$ satisfies a polynomial-type CLT with normalization $s_n=(n^d\{np(1-p)\})^{-1/2}$.
Furthermore, Lemma~\ref{lem:cons_condi1} shows that $\{H_n\}_{n>d}$ satisfies Condition~(1) in Definition~\ref{df:concentration} with the normalization $s_n$.
Hence, it only remains to check Condition~(2) in Definition~\ref{df:concentration} for applying Theorem~\ref{thm:CLT_poly-diff}.
For this purpose, a more precise estimate of $|\cW_{k,s}|$ than that given in Lemma~\ref{lem:cW_ks}(2) is crucial.
Such estimate was investigated in~\cite[Section~5]{KR17} involving a higher-dimensional generalization of the classical notion of FK sentences~\cite[Section~2.1.6]{AGZ11}.
Although we adopt slightly different definitions of the $(n,d)$-words and equivalence relation among them from those in~\cite{KR17}, almost the same estimate as~\cite[Proposition~5.3]{KR17} still holds by an appropriate modification of the proof.
\begin{prop}\label{prop:KR17_prop5.3}
For every $k\ge2$ and $d+1\le s\le\lfloor k/2\rfloor+d$, it holds that
\[
|\cW_{k,s}|\le C_d(2\sqrt d)^k\sum_{m=0}^{k-2(s-d)}\frac{(C_dk^3)^m}{m!}.
\]
Here, $C_d$ is a constant depending only on $d$.
\end{prop}
%

Using Proposition~\ref{prop:KR17_prop5.3}, we will prove Condition~(2) in Definition~\ref{df:concentration} with normalization $s_n=(n^d\{np(1-p)\})^{-1/2}$.
\begin{lem}\label{lem:cons_condi2}
Assume that $np(1-p)=\om((\log n)^4)$.
Then, for any measurable function $g\colon\R\to\R$ of polynomial growth with $\Supp(g)\subset[-2\sqrt d,2\sqrt d]^c$, it holds that
\[
\lim_{n\to\infty}n^d\{np(1-p)\}\E[\langle L_{H_n},g\rangle^2]=0.
\]
\end{lem}
\begin{proof}
Noting that $np(1-p)=\om((\log n)^4)$, we take an integer-valued function $k(n)$ so that $k(n)=\om(\log n)$ and $k(n)=o(\{np(1-p)\}^{1/4})$.
Let $g\colon\R\to\R$ be a measurable function of polynomial growth with $\Supp(g)\subset[-2\sqrt d,2\sqrt d]^c$.
Noting that $\Supp(g)$ is a closed set, we take $\eps>0$ such that $\Supp(g)\subset[-2\sqrt d-\eps,2\sqrt d+\eps]^c$.
Since $g$ is of polynomial growth, there exist constants $c_1,c_2\in\N$ such that
\begin{equation}\label{eq:condi2_cons_1}
|g(x)|\le c_1|x|^{c_2}
\end{equation}
holds for any $|x|>2\sqrt d+\eps$.
Also, we can take $N\in\N$ such that $n\ge N$ implies that
\begin{equation}\label{eq:condi2_cons_2}
\biggl(\frac{|x|}{2\sqrt d+\eps/2}\biggr)^{k(n)}\ge|x|^{c_2}
\end{equation}
for any $|x|>2\sqrt d+\eps$.
We may also assume that $n\ge N$ implies $np(1-p)\ge2$ for later use.
Then, for $n\ge N$, we have
\begin{align*}
n^d\{np(1-p)\}\E[\langle L_{H_n},g\rangle^2]
&=\frac{n^d\{np(1-p)\}}{\binom nd^2}\E\biggl[\biggl(\sum_{i=1}^{\binom nd}g(\lm_i[H_n])\biggr)^2\biggr]\\
&\le\frac{n^d\{np(1-p)\}}{\binom nd}\E\biggl[\sum_{i=1}^{\binom nd}g(\lm_i[H_n])^2\biggr]\\
&\le\frac{c_1^2n^d\{np(1-p)\}}{\binom nd}\E\biggl[\sum_{i=1}^{\binom nd}\lm_i[H_n]^{2c_2}\1_{\{|\lm_i[H_n]|>2\sqrt d+\eps\}}\biggr]\\
&\le\frac{c_1^2n^d\{np(1-p)\}}{\binom nd}\E\biggl[\sum_{i=1}^{\binom nd}\biggl(\frac{\lm_i[H_n]}{2\sqrt d+\eps/2}\biggr)^{2k(n)}\biggr]\\
&=\frac{c_1^2n^d\{np(1-p)\}}{(2\sqrt d+\eps/2)^{2k(n)}}\E\langle L_{H_n},x^{2k(n)}\rangle.
\end{align*}
The second line follows from the Cauchy--Schwarz inequality.
To derive the third line, we used $\Supp(g)\subset[-(2\sqrt d+1),2\sqrt d+1]^c$ and~\eqref{eq:condi2_cons_1}.
The fourth line is obtained by~\eqref{eq:condi2_cons_2}.

Now, let $n\ge N$ and $k\in\N$ be fixed.
From~\eqref{eq:KR17_Lem3.2_1} and~\eqref{eq:KR17_Lem3.2_2}, we have
\begin{align*}
\E\langle L_{H_n},x^{2k}\rangle
=\sum_{s=d+1}^{k+d}\sum_{w\in\cW_{2k,s}}\frac{(n-d)(n-d-1)\cdots(n-s+1)}{\{np(1-p)\}^k}\bar T_n(w)
\le\sum_{s=d+1}^{k+d}\frac{|\cW_{2k,s}|}{\{np(1-p)\}^{k+d-s}}.
\end{align*}
From Proposition~\ref{prop:KR17_prop5.3}, the right-hand side of the above equation is bounded above by
\begin{align*}
&C_d(2\sqrt d)^{2k}\sum_{s=d+1}^{k+d}\sum_{m=0}^{2k-2(s-d)}\frac{(8C_dk^3)^m}{m!\{np(1-p)\}^{k+d-s}}\\
&=C_d(2\sqrt d)^{2k}\sum_{m=0}^{2(k-1)}\sum_{s=d+1}^{k+d-\lceil m/2\rceil}\frac{(8C_dk^3)^m}{m!\{np(1-p)\}^{k+d-s}}\\
&\le C_d(2\sqrt d)^{2k}\frac{np(1-p)}{np(1-p)-1}\sum_{m=0}^{2(k-1)}\frac1{m!}\biggl(\frac{8C_dk^3}{\sqrt{np(1-p)}}\biggr)^m\\
&\le C_d(2\sqrt d)^{2k}\frac{np(1-p)}{np(1-p)-1}\exp\biggl(\frac{8C_dk^3}{\sqrt{np(1-p)}}\biggr)\\
&\le2C_d(2\sqrt d)^{2k}\exp\biggl(\frac{8C_dk^3}{\sqrt{np(1-p)}}\biggr).
\end{align*}
For the last line, we used $np(1-p)\ge2$.
Combining the above estimates yields
\begin{align*}
&n^d\{np(1-p)\}\E[\langle L_{H_n},g\rangle^2]\\
&\le\frac{c_1^2n^d\{np(1-p)\}}{(2\sqrt d+\eps/2)^{2k(n)}}2C_d(2\sqrt d)^{2k(n)}\exp\biggl(\frac{8C_dk(n)^3}{\sqrt{np(1-p)}}\biggr)\\
&\le2C_dc_1^2n^{d+1}\biggl(1+\frac\eps{4\sqrt d}\biggr)^{-2k(n)}\exp\biggl(\frac{8C_dk(n)^3}{\sqrt{np(1-p)}}\biggr)\\
&=\exp\biggl(\log(2C_dc_1^2)+(d+1)\log n+2k(n)\biggr\{\frac{4C_dk(n)^2}{\sqrt{np(1-p)}}-\log\biggl(1+\frac\eps{4\sqrt d}\biggr)\biggr\}\biggr)
\end{align*}
for $n\ge N$.
Obviously, the right-hand side converges to zero as $n\to\infty$ since $k(n)=\om(\log n)$ and $k(n)^2=o\bigl(\sqrt{np(1-p)}\bigr)$.
\end{proof}

Finally, we prove Theorem~\ref{thm:CLT_diff}.
\begin{proof}[Proof of Theorem~\ref{thm:CLT_diff}]
As seen in Remark~\ref{rem:poly-type_CLT}, Lemmas~\ref{lem:cons_condi1} and~\ref{lem:cons_condi2}, the sequence $\{H_n\}_{n>d}$ satisfies both a polynomial-type CLT and the convex concentration property with normalization $s_n=(n^d\{np(1-p)\})^{-1/2}$.
Furthermore, the constant $M$ in Condition~(2) in Definition~\ref{df:concentration} can be taken as $2\sqrt d$.
Thus, the conclusion follows immediately from Theorem~\ref{thm:CLT_poly-diff}.
\end{proof}


\section*{Acknowledgements}
The first author is supported by a JSPS Grant-in-Aid for Scientific Research (A) (JP20H00119).

\begin{footnotesize}
\bibliographystyle{spmpsci}
\bibliography{bib}
\end{footnotesize}

\end{document}